\documentclass{amsart}[11pt]

\usepackage[all]{xy}
\usepackage{amssymb,mathrsfs,amscd,graphicx,array, multirow,longtable,
enumerate, amsfonts, euscript}

\usepackage{hyperref}
\makeatletter
\def\@tocline#1#2#3#4#5#6#7{\relax
  \ifnum #1>\c@tocdepth 
  \else
    \par \addpenalty\@secpenalty\addvspace{#2}%
    \begingroup \hyphenpenalty\@M
    \@ifempty{#4}{%
      \@tempdima\csname r@tocindent\number#1\endcsname\relax
    }{%
      \@tempdima#4\relax
    }%
    \parindent\z@ \leftskip#3\relax \advance\leftskip\@tempdima\relax
    \rightskip\@pnumwidth plus4em \parfillskip-\@pnumwidth
    #5\leavevmode\hskip-\@tempdima
      \ifcase #1
       \or\or \hskip 1em \or \hskip 2em \else \hskip 3em \fi%
      #6\nobreak\relax
    \dotfill\hbox to\@pnumwidth{\@tocpagenum{#7}}\par
    \nobreak
    \endgroup
  \fi}
\makeatother

\usepackage{letltxmacro}
\LetLtxMacro{\oldsqrt}{\sqrt}
\renewcommand{\sqrt}[2][]{\,\oldsqrt[#1]{#2}\,}

\def\bmu{\boldsymbol \mu}

\newcommand{\abs}[1]{\lvert #1 \rvert}
\newcommand{\zmod}[1]{\mathbb{Z}/ #1 \mathbb{Z}}
\newcommand{\umod}[1]{(\mathbb{Z}/ #1 \mathbb{Z})^\times}

\newcommand{\dangle}[1]{\left\langle #1 \right\rangle}

\DeclareMathSymbol{\twoheadrightarrow} {\mathrel}{AMSa}{"10}

\DeclareMathOperator{\ord}{ord}

\DeclareMathOperator{\Pic}{Pic}

\DeclareMathOperator{\Aut}{Aut}
\DeclareMathOperator{\End}{End}
\DeclareMathOperator{\Hom}{Hom}

\DeclareMathOperator{\Gal}{Gal}
\DeclareMathOperator{\Mat}{Mat}

\DeclareMathOperator{\Nm}{N}  

\DeclareMathOperator{\Lie}{Lie}
\DeclareMathOperator{\GL}{GL}

\DeclareMathOperator{\Sp}{Sp}

\def\a{{\mathfrak a}} 
\def\c{{\mathfrak c}} 
\def\d{{\mathfrak d}} 

\newcommand{\scrD}{\mathscr{D}}

\newcommand{\ff}{\mathbb{F}}

\newcommand{\qq}{\mathbb{Q}}

\newcommand{\zz}{\mathbb{Z}}

\newcommand{\calO}{\mathcal{O}}
\newcommand{\calE}{\mathcal{E}}

\DeclareMathOperator{\Cl}{Cl}
\DeclareMathOperator{\Isog}{Isog}
\DeclareMathOperator{\Mass}{Mass}

\newcommand{\wt}{\widetilde}

\newcommand{\dieu}{Dieudonn\'{e} }

\def\makeop#1{\expandafter\def\csname#1\endcsname
  {\mathop{\rm #1}\nolimits}\ignorespaces}
\makeop{Hom}   \makeop{End}   \makeop{Aut}   \makeop{Isom}  \makeop{Pic} 
\makeop{Gal}   \makeop{ord}   \makeop{Char}  \makeop{Div}   \makeop{Lie} 
\makeop{PGL}   \makeop{Corr}  \makeop{PSL}   \makeop{sgn}   \makeop{Spf}
\makeop{Spec}  \makeop{Tr}    \makeop{Nr}    \makeop{Fr}    \makeop{disc}
\makeop{Proj}  \makeop{supp}  \makeop{ker}   \makeop{im}    \makeop{dom}
\makeop{coker} \makeop{Stab}  \makeop{SO}    \makeop{SL}    \makeop{SL}
\makeop{Cl}    \makeop{cond}  \makeop{Br}    \makeop{inv}   \makeop{rank}
\makeop{id}    \makeop{Fil}   \makeop{Frac}  \makeop{GL}    \makeop{SU}
\makeop{Nrd}   \makeop{Sp}    \makeop{Tr}    \makeop{Trd}   \makeop{diag}
\makeop{Res}   \makeop{ind}   \makeop{depth} \makeop{Tr}    \makeop{st}
\makeop{Ad}    \makeop{Int}   \makeop{tr}    \makeop{Sym}   \makeop{can}
\makeop{length}\makeop{SO}    \makeop{torsion} \makeop{GSp} \makeop{Ker}
\makeop{Adm}   \makeop{Mat}
\def\makebb#1{\expandafter\def
  \csname bb#1\endcsname{{\mathbb{#1}}}\ignorespaces}
\def\makebf#1{\expandafter\def\csname bf#1\endcsname{{\bf
      #1}}\ignorespaces} 
\def\makegr#1{\expandafter\def
  \csname gr#1\endcsname{{\mathfrak{#1}}}\ignorespaces}
\def\makescr#1{\expandafter\def
  \csname scr#1\endcsname{{\EuScript{#1}}}\ignorespaces}
\def\makecal#1{\expandafter\def\csname cal#1\endcsname{{\mathcal
      #1}}\ignorespaces} 

\def\doLetters#1{#1A #1B #1C #1D #1E #1F #1G #1H #1I #1J #1K #1L #1M
                 #1N #1O #1P #1Q #1R #1S #1T #1U #1V #1W #1X #1Y #1Z}
\def\doletters#1{#1a #1b #1c #1d #1e #1f #1g #1h #1i #1j #1k #1l #1m
                 #1n #1o #1p #1q #1r #1s #1t #1u #1v #1w #1x #1y #1z}
\doLetters\makebb   \doLetters\makecal  \doLetters\makebf
\doLetters\makescr 
\doletters\makebf   \doLetters\makegr   \doletters\makegr

\def\ol{\overline}
\def\wt{\widetilde}

\def\ul{\underline}

\newcommand{\isoto}{\stackrel{\sim}{\longrightarrow}}
\newcommand{\embed}{\hookrightarrow}

\def\Fpbar{\overline{\bbF}_p}
\def\Fqbar{\overline{\bbF}_q}
\def\Fp{{\bbF}_p}
\def\Fq{{\bbF}_q}

\def\Qp{{\bbQ}_p}

\def\Qbar{\overline{\bbQ}}

\def\ac{algebraically closed }
\def\ch{characteristic }

\newcommand{\Z}{\mathbb Z}
\newcommand{\Q}{\mathbb Q}

\newcommand{\C}{\mathbb C}
\newcommand{\F}{\mathbb F}

\newcommand{\npr}{\noindent }

\newcounter{thmcounter} 
\numberwithin{thmcounter}{section}
\newtheorem{thm}[thmcounter]{Theorem}
\newtheorem{lem}[thmcounter]{Lemma}
\newtheorem{lemma}[thmcounter]{Lemma}
\newtheorem{cor}[thmcounter]{Corollary}
\newtheorem{prop}[thmcounter]{Proposition}
\theoremstyle{definition}
\newtheorem{defn}[thmcounter]{Definition}

\newtheorem{rem}[thmcounter]{Remark}
\newtheorem{remark}[thmcounter]{Remark}

\numberwithin{equation}{section}

\newtheoremstyle{notitle}  
  {}
  {}
  {\itshape}
  {}
  {}
  {\ }
  {.5em}
  {}
\theoremstyle{notitle}

\title[Superspecial abelian varieties]{On superspecial 
  abelian varieties over finite fields}
\author{Jiangwei Xue,Tse-Chung Yang and Chia-Fu Yu}

\address{(Xue) Collaborative Innovation Centre of Mathematics, 
School of Mathematics and Statistics, Wuhan University, Luojiashan,
 Wuhan, Hubei, 430072, P.R. China.}
\email{xue\_j@whu.edu.cn}

\address{(Yang) Institute of Mathematics, Academia Sinica,
  Astronomy-Mathematics Building, 6F, No. 1, Sec. 4, Roosevelt Road,
  Taipei 10617, TAIWAN.} 
\email{tsechung@math.sinica.edu.tw}

\address{(Yu) Institute of Mathematics,
  Academia Sinica and NCTS, Astronomy-Mathematics
  Building, No. 1, Sec. 4, Roosevelt Road, Taipei 10617, TAIWAN.}
\email{chiafu@math.sinica.edu.tw} 

\begin{document}

\date{\today} 

\begin{abstract}
In this paper we establish a new lattice description for superspecial
abelian varieties over a finite field $\Fq$ of $q=p^a$ elements. 
Our description depends on the parity of the exponent $a$ of $q$. 
When $q$ is an odd power of the prime $p$, we give 
an explicit formula for the number of superspecial abelian 
surfaces over $\Fq$.


\end{abstract}

\maketitle



\section{Introduction}
\label{sec:intro}

Throughout this paper $p$ denotes a prime
number, and $q=p^a$ a power of $p$ with an exponent $a\in \bbN$, the
set of strictly positive integers. The goal of this paper is to calculate explicitly the number of
superspecial abelian surfaces over a finite field $\Fq$.
This can be regarded as a natural extension of 
works of the authors 
\cite{xue-yang-yu:num_inv, xue-yang-yu:ECNF} and 
the last named author \cite{yu:sp_prime} contributed to 
the study of supersingular abelian varieties over finite fields.



Recall that an abelian variety over a field $k$ of
characteristic $p$ is said to be {\it supersingular} if it is
isogenous to a product of supersingular elliptic curves over an
algebraic closure 
$\bar k$ of $k$; it is said to be {\it superspecial} if it is
isomorphic to a product of supersingular elliptic curves over $\bar
k$. As any supersingular abelian variety is isogenous to a
superspecial abelian variety, it is very common to study
supersingular abelian varieties through investigating 
the classification of superspecial abelian varieties. 






For any integer $d\ge 1$, let $\Sp_d(\F_q)$
denote the set of isomorphism classes of $d$-dimensional 
superspecial abelian varieties over the finite field $\F_q$ of $q$ 
elements. The case where $d=1$ concerns the classification of
supersingular elliptic curves over finite fields. 
The theory of elliptic curves over finite fields 
has been studied by Deuring since 1940's and becomes well known. 
There are explicit descriptions for each isogeny class; 
see Waterhouse \cite[Section 4]{waterhouse:thesis}.
However, the authors could not find an explicit formula for
$|\Sp_1(\Fq)|$ in the literature.   
For the sake of completeness we include a formula for
$|\Sp_1(\Fq)|$,  
based on the exposition of Deuring's results by Waterhouse
\cite{waterhouse:thesis}. 
The goal of the present paper is then to find an explicit 
formula for the number $|\Sp_d(\Fq)|$ in the case where $d=2$. 

Before stating our main results, we describe a basic method of
counting $\Sp_d(\Fq)$. For simplicity assume that $\F_q=\Fp$ is
the prime finite field for the moment. 
One can divide the finite set $\Sp_d(\F_p)$ into finitely many subsets
according to the isogeny classes of members. 
Therefore, it suffices to classify
all $d$-dimensional supersingular isogeny classes and to count
the number of superspecial members in each supersingular 
isogeny class. The Honda-Tate theorem allows us to describe
isogeny classes over $\Fq$ 
in terms of multiple Weil $q$-numbers (which are simply finite 
nonnegative integral formal sums of Weil $q$-numbers up to conjugate;
see Section~\ref{sec:curve.1}). 
If $\pi$ is a supersingular multiple Weil $q$-number,
we denote by $[X_\pi]$ the corresponding supersingular isogeny class
(here $X_\pi$ is an abelian variety in this class), 
$H(\pi)$ the number of isomorphism classes of abelian varieties in
$[X_\pi]$ and $H_{sp}(\pi)$ the number of isomorphism classes of 
{\it superspecial} abelian varieties in $[X_\pi]$. Then we have
\begin{equation}
  \label{eq:intro.1}
  |\Sp_d(\Fp)|=\sum_{\pi} H_{sp}(\pi)
\end{equation}
where $\pi$ runs through all supersingular multiple Weil $p$-numbers with
$\dim X_\pi=d$. We classify all possible isogeny 
classes $\pi$'s occurring in the sum (see Sections 2--3).  
The problem then is to compute each term
$H_{sp}(\pi)$.  One should
distinguish the cases according to whether the endomorphism algebra
$\End^0(X_\pi)=\End(X_\pi)\otimes \Q$ of $X_\pi$ 
satisfies the Eichler condition \cite[Section~III.4, p.81]{vigneras} or not. 
We now focus on the case where $d=2$.

Consider the case where $\pi$ is the Weil $p$-number
$\sqrt{p}$. Correspondingly,  $X_\pi$ is a supersingular abelian
surface. It is known (see Tate \cite{tate:eav}) that 
the endomorphism algebra $\End^0(X_\pi)$ of $X_\pi$ is isomorphic to 
the totally definite quaternion algebra algebra
$D=D_{\infty_1,\infty_2}$ over the quadratic real field 
$F=\Q(\sqrt{p})$ ramified exactly at 
the two real places $\{\infty_1,\infty_2\}$ of $F$.  
In this case all abelian surfaces in the isogeny
class $[X_{\sqrt{p}}]$ are superspecial,
i.e.~$H(\sqrt{p})=H_{sp}(\sqrt{p})$. When $p=2$ or $p\equiv 3 \pmod
4$, Waterhouse proved that the
number $H(\sqrt{p})$ is equal to the class number
$h(D_{})$ of $D_{}$. 
When $p \equiv 1 \pmod 4$, the number
$H(\sqrt{p})$ is equal to the sum of $h(D_{})$ 
and the class numbers of two other proper $\Z[\sqrt{p}]$-orders in
$D_{}$ of index $8$ and $16$, respectively 
(the descriptions of these orders are made concrete by results of \cite{yu:smf}). 
These class numbers are computed systematically in our previous work \cite{xue-yang-yu:ECNF}. 
Therefore,  we obtain an explicit formula for the term
$H_{sp}(\sqrt{p})$ given below. 
In what follows we write $K_{m,j}$ for the number field 
$\Q(\sqrt{m},\sqrt{-j})$ for any square-free integers $m>1$ and $j\ge
1$. If  $m\equiv 1 \pmod 4$, then we define
\begin{equation}
  \label{eq:varpi_d}
  \varpi_m:=3 [O_{\Q(\sqrt{m})}^\times: \Z[\sqrt{m}]^\times]^{-1},
\end{equation}
where $O_{\Q(\sqrt{m})}$ denotes the ring of integers of $\Q(\sqrt{m})$. 
By similar arguments as those in \cite[Lemma~4.1 and
Section~4.2]{xue-yang-yu:num_inv}, we have $\varpi_m\in\{1,3\}$, and 
$\varpi_m=3$ if $m\equiv 1\pmod 8$. The class number of a number field
$K$ is denoted by $h(K)$. When $K=\Q(\sqrt{m})$, we write
$h(\sqrt{m})$ for $h(\Q(\sqrt{m}))$ instead. 

\begin{thm}\label{1.2}
  Let $H(\sqrt{p})$ be the number of $\ff_p$-isomorphism 
  classes of abelian
  varieties in the simple isogeny class corresponding to the Weil
  $p$-number $\pi=\sqrt{p}$, and let $F=\Q(\sqrt{p})$. Then\\
(1)  $H(\sqrt{p})=1,2,3$ for $p=2,3, 5$, respectively. \\
(2)  For $p>5$ and $p\equiv 3 \pmod 4$, we have
    \begin{equation}
      \label{eq:intro.2}
     H(\sqrt{p})=\frac{1}{2}h(F)\zeta_F(-1) +
     \left(\frac{3}{8}+\frac{5}{8}\left(2-\left(\frac{2}{p}\right)
    \right)\right)h(K_{p,1})+\frac{1}{4}h(K_{p,2})+
    \frac{1}{3}h(K_{p,3}),    
\end{equation}
where $\zeta_F(s)$ is the Dedekind zeta function
of $F$.\\
(3) For $p>5$ and $p\equiv 1 \pmod 4$, we have
    \begin{equation}
      \label{eq:1.2}
   H(\sqrt{p})=       
      \begin{cases}
        8 \zeta_F(-1)h(F)+ h(K_{p,1})+\frac{4}{3}
        h(K_{p,3}) & \text{for $p\equiv 1 \pmod 8$;} \\
      \frac{1}{2}(15\varpi_p+1) \zeta_F(-1)h(F)+
        \frac{1}{4}(3\varpi_p+1) h(K_{p,1})
      +\frac{4}{3} h(K_{p,3}) & \text{for $p\equiv 5 \pmod 8$;} \\
      \end{cases} 
   \end{equation}
\end{thm}

The computation in Theorem~\ref{1.2} is based on 
the generalized Eichler class
formula \cite[Theorem 1.4]{xue-yang-yu:ECNF} 
that the authors developed. This formula allows us to compute the class
number of an arbitrary $\Z$-order in a totally definite quaternion
over a totally real field $F$. This $\Z$-order does not necessarily contains the maximal
order $O_F$ of $F$. For a quadratic real field $F$, the special zeta value $\zeta_F(-1)$ 
 can be calculated by 
Siegel's formula \cite[Table 2, p. 70]{Zagier-1976-zeta} 
\begin{equation}
  \label{eq:intro.35}
  \zeta_F(-1)=\frac{1}{60}\sum_{\substack{b^2+4ac=\d_F\\ a,c>0}} a,
\end{equation}
where  $\d_F$ is the discriminant of $F/\Q$, $b\in \zz$ 
and $a,c\in \bbN$.   


The first main result of this paper gives the following explicit formula 
for $\abs{\Sp_2(\F_p)}$, the number of isomorphism classes of superspecial abelian surfaces over $\Fp$. 
To obtain this formula, we calculate all  terms $H_{sp}(\pi)$ with 
$\pi\neq \pm \sqrt{p}$ in
(\ref{eq:intro.1}),  and then sum them up together with $H(\sqrt{p})$. The computation of
$H_{sp}(\pi)$ uses a lattice description for superspecial abelian
varieties; see Section 5 for details.  
Similar to Theorem~\ref{1.2}, special 
attentions have to be paid to the cases with small primes $p$. 


\begin{thm}\label{1.3}
  We have $|\Sp_2(\Fp)|=H(\sqrt{p})+\Delta(p)$, where the formula for
  $H(\sqrt{p})$ is stated in Theorem~\ref{1.2} and $\Delta(p)$ is the
  number described as follows. 
  \begin{enumerate}
  \item $\Delta(p)=15,20,9$ for $p=2,3,5$,
  respectively.
  \item For $p>5$ and $p\equiv 1 \pmod 4$, we have 
  \begin{equation}
    \label{eq:intro.4}
    \Delta(p)=(\varpi_p+1) h(K_{p,3})+h(K_{2p,1})
  +h(K_{3p,3})+h(\sqrt{-p}),
  \end{equation}

  \item For $p>5$ and $p\equiv 3 \pmod 4$, we have 
  \begin{equation}
    \label{eq:intro.5}
   \Delta(p)=h(K_{p,3})+h(K_{2p,1})+(\varpi_{3p}+1)
   h(K_{3p,3})+\left(4-\left (\frac{2}{p}\right ) \right )
   h(\sqrt{-p}).
  \end{equation}
  \end{enumerate}
\end{thm}


A key ingredient of our computation for $\Sp_2(\Fp)$ is
Theorem~\ref{prop:sp.1}, which works only for the prime finite
fields. Centeleghe and Stix \cite{MR3317765} provide a categorical
description of Theorem~\ref{prop:sp.1} (also compare \cite[Theorem
3,1]{yu:sp_prime}). However, their results are also limited to the
prime finite fields. When the base field $\Fq$ is no longer the prime
finite field, direct calculations via the counting method described
earlier for $\Sp_d(\Fq)$ (even when $d=2$) become more complicated.

Our second main result extends the computations of  $\Sp_2(\Fp)$ to
$\Sp_2(\Fq)$ for more general finite fields $\Fq$ via Galois cohomology. 
Observe that if $d>1$, then there is only one isomorphism class of 
$d$-dimensional superspecial abelian varieties over $\Fpbar$ (see
\cite[Section 1.6, p.~13]{li-oort} or Theorem~\ref{gal.sp}). 
Suppose $X_0$ is any $d$-dimensional 
superspecial abelian variety over $\Fp$. 
Then there is a bijection of finite sets
\begin{equation}
  \label{eq:intro.6}
  \Sp_d(\Fp)\simeq H^1(\Gamma_{\Fp}, G), \quad d>1,
\end{equation}
where $\Gamma_{\Fp}=\Gal(\Fpbar/\Fp)$ is the absolute Galois group
of $\Fp$, and $G=\Aut(X_0\otimes \Fpbar)$. 
Thus, computing the Galois
cohomology would lead to a second proof of Theorem~\ref{1.3}.
However, the complexity of the final formula as in Theorem~\ref{1.3}
suggests that the computation of this Galois cohomology is likely on
the same level of difficulty as the counting method via (\ref{eq:intro.1}). 
However, the true advantages of  connecting to Galois cohomology are two
folds.
\begin{enumerate} 
\item[(a)] It naturally relates $\Sp_d(\Fq)$ and $\Sp_d(\F_{q'})$ in
  the sense of Theorem~\ref{1.4} when the
  exponents in $q=p^a$ and $q'=p^{a'}$
  have the same parity. 
\item[(b)] It gives rise to a lattice description for
  $\Sp_d(\Fq)$ when $q=p^a$ is an even power of $p$; see
  Theorem~\ref{gal_coh.3}.  
\end{enumerate}


\begin{thm}\label{1.4}
Let $q$ and $q'$ be powers of $p$ with same
exponent parity and $d\ge 1$ an integer. Then there is a natural
bijection $\Sp_d(\Fq)\simeq \Sp_d(\F_{q'})$ preserving isogeny
classes. In particular, the same formulas in Theorem~\ref{1.3} hold
 for $|\Sp_2(\Fq)|$ since $|\Sp_d(\Fq)|=|\Sp_d(\Fp)|$ when $q$ is an
odd power of $p$. 
\end{thm}
The bijection for the case $d=1$ is handled separately in
Section~\ref{sec:curve} (see Remark~\ref{curve.4}). For $d\geq 2$, the
bijection is established in Theorem~\ref{gal.1}.  
Along the way,  we prove in Section~\ref{subsec:galois-descent} the
following 
general result connecting isogeny classes of abelian varieties 
over $\F_q$ with cohomology classes. 
\begin{thm}\label{thm:intro-descent-isog}
Let $[X_0]$ be the $\F_q$-isogeny class of an arbitrary abelian
variety $X_0$ over $\F_q$, and $G_\Q=\End^0(\ol{X}_0)^\times$ where $\ol X_0=X_0\otimes_{\F_q}\Fqbar$. We write $E^0(\Fqbar/\F_q,
[X_0])$ for the set of $\F_q$-isogeny classes of abelian varieties $[X]$
 such that $\ol X$ is isogenous to $\ol X_0$ over $\Fqbar$. Then there is a canonical bijection
 of pointed sets
\[ E^0(\Fqbar/\F_q,
[X_0]) \isoto H^1(\Gamma_{\F_q}, G_\Q)  \]
sending $[X_0]$ to the trivial cohomology class.
\end{thm}

Theorem~\ref{1.4} together with Theorem~\ref{prop:sp.1} give a
new lattice description in Theorem~\ref{gal.odd} for $\Sp_d(\Fq)$ when $q$ is an odd power of
$p$. When $q$ is an even power of $p$, a
lattice description of $\Sp_d(\F_q)$ completely different from the odd
case is given in Theorem~\ref{gal_coh.3}, which paves
the way to explicit formulas of $\abs{\Sp_2(\Fq)}$. 
The detailed formulas and computations will be presented in a separated paper \cite{xue-yang-yu:conj_finite}.

The paper is organized as follows.  In Section~\ref{sec:par}, we
parameterize simple isogeny classes of supersingular abelian varieties
over $\F_q$ using Weil $q$-numbers. Their dimensions are
calculated in Section~\ref{sec:dim}. In Section~\ref{sec:curve} we
treat the dimension 1 case and calculate the the number of isomorphism
classes of supersingular elliptic curves over finite fields. The
dimension 2 case is then treated in Section~\ref{sec:sp}, except we work
exclusively over the prime field $\F_p$, and some arithmetic calculations are
postponed to Section~\ref{sec:arithmetic-results}.
Section~\ref{sec:gal_coh} studies the parity property via Galois
cohomology, thus providing means to extend results of
Section~\ref{sec:sp} to all $\F_{p^a}$ with $a$ odd.  The
aforementioned lattices descriptions are obtained
in this process.

\section{Parameterization of supersingular isogeny classes}
\label{sec:par}

\subsection{}
\label{sec:par.1}

Let $q=p^a$ be a power of a prime number $p$. In this section 
we parameterize simple isogeny classes of supersingular abelian varieties
over $\F_q$. 
Let $\Qbar\subset \C$ be the algebraic closure of $\Q$ in $\C$.
If two algebraic numbers $\alpha, \beta \in \Qbar$ are conjugate over $\Q$, then
we write $\alpha\sim \beta$. Recall that an algebraic integer $\pi\in \Qbar$ 
is said to be a {\it Weil $q$-number} if $|\iota(\pi)|=q^{1/2}$ for any 
embedding $\iota: \Q(\pi)\embed \C$.
By the Honda-Tate theory, the simple isogeny classes of
abelian varieties over $\F_q$ are in bijection with the conjugacy
classes of Weil $q$-numbers.  A Weil $q$-number is said to be {\it
  supersingular} if the corresponding
isogeny class consists of supersingular abelian varieties. Let $W_q^{\rm ss}$ denote the set of conjugacy
classes of supersingular Weil $q$-numbers. We will find a unique
representative for each conjugacy class in $W_q^{\rm ss}$.

Let $\pi$ be a supersingular Weil $q$-number.
It is known (the Manin-Oort Theorem, cf.~\cite[Theorem 2.9]{yu:QMav}) 
that $\pi=\sqrt{q} \zeta$ for a root of unity $\zeta$.
Let $K:=\Q(\pi)$ and $L:=\Q(\sqrt{q}, \zeta)$. Note that 
both $L$ and $K$ are
abelian extensions over $\Q$. 
For any $n\in \bbN$ (the set of positive integers), write
$\zeta_n:=e^{2\pi i/n}\in \Qbar$.  

\begin{lemma}\label{lemma:par.1}
  Any supersingular Weil $q$-number $\pi$ 
  is conjugate to $\sqrt{q}\zeta_n$ or 
  $-\sqrt{q}\zeta_n$ with $n\not \equiv 2 \pmod 4$. 
\end{lemma}
\begin{proof}
  Let $\pi=\sqrt{q} \zeta_m^\nu$ for some positive integers $\nu$ and
  $m$ with $(\nu,m)=1$. Choose an element $\sigma\in \Gal(L/\Q)$ such
  that 
  $\sigma(\zeta_m^\nu)=\zeta_m$, Then $\sigma(\pi)=\pm \sqrt{q}
  \zeta_m$. 

  If $m \not \equiv 2 \pmod 4$, then we are done. Suppose that $m=2k$
  for an odd integer $k=1-2u$. Clearly $(k,u)=1$. Since
  $\zeta_{2k}=\zeta_{2k}^{k+2u}=-\zeta_{2k}^{2u}=-\zeta_k^u$, we
  have 
\[ \pm \sqrt{q} \zeta_{2k}=\mp \sqrt{q} \zeta_{k}^u \sim \epsilon
  \sqrt{q} \zeta_k, \quad \text{for some\ } \epsilon\in \{\pm 1\} \]
by the previous argument. 
\end{proof}

By Lemma~\ref{lemma:par.1}, 
there is a unique subset $W$ of $\{\pm \sqrt{q}
\zeta_n ; n\not \equiv 2 \pmod 4\}$ that contains $\{\, \sqrt{q}
\zeta_n ; n\not \equiv 2 \pmod 4\}$ and represents $W_q^{\rm ss}$.
We often identify $W$ with $W_q^{\rm ss}$.  
To determine the set $W_q^{\rm ss}$, we need to characterizes 
when $\sqrt{q} \zeta_n$ and $-\sqrt{q} \zeta_n$ 
are conjugate.  

As usual, the Galois group $G_n:=\Gal(\Q(\zeta_n)/\Q)$ is naturally 
identified with $(\Z/n\Z)^\times$ by mapping any $r\in (\Z/n\Z)^\times$
to the element $\sigma_r\in G_n$ with $\sigma_r(\zeta_n)=\zeta_n^r$.  


\subsection{}
\label{sec:par.2}
Let us first assume that $a$ is even, i.e., $\sqrt{q}\in \Q$. Then
$\sqrt{q} \zeta_n\sim -\sqrt{q} \zeta_n$ if and only if there is an
element $\sigma_r\in G_n$ such that $\sigma_r(\zeta_n)=-\zeta_n$. It
is easy to see that
\begin{equation}
  \label{eq:even.1}
 \zeta_n^r=-\zeta_n \iff 2|n \text{ and } r=\frac{n}{2}+1,  
\end{equation}
and if $4|n$, then $(r,n)=1$. As $n\not\equiv 2 \pmod 4$, this gives 
\begin{equation}
  \label{eq:even.2}
\sqrt{q} \zeta_n\sim -\sqrt{q} \zeta_n \iff 4|n.  
\end{equation}
Thus,
\begin{equation}
  \label{eq:W_even}
  W_q^{\,\rm ss}\simeq \{\pm \sqrt{q} \zeta_n\ ;\  2\nmid n\,\} \cup 
\{\sqrt{q} \zeta_n\ ;\ 4|n\,\}.   
\end{equation}
Alternatively, since $\sqrt{q}\in\Q$, we have $\sqrt{q}\zeta_n^\nu\sim
\sqrt{q}\zeta_n$ for any $\nu\in \bbN$ with $(\nu, n)=1$. 
It follows that 
\begin{equation}
  \label{eq:W_even.1}
  W_q^{\,\rm ss}\simeq \{\sqrt{q} \zeta_n\ ;\  n\in \bbN \}. 
\end{equation}
The two descriptions (\ref{eq:W_even}) and (\ref{eq:W_even.1}) match, because when $n$ is odd, $-\zeta_n$ is a primitive $2n$-th root of
unity and hence $-\sqrt{q}\zeta_n$ is conjugate to
$\sqrt{q}\zeta_{2n}$. 
\subsection{}
\label{sec:par.3}
We now assume that $a$ is odd. Let $\d_p$ be the
discriminant of $\Q(\sqrt{p})$. In other words,  $\d_p=p$ if $p\equiv 1\pmod{4}$, otherwise $\d_p=4p$. By \cite[Chapter V,
Theorem~48]{ANT-Frohlich-Taylor},  
$\sqrt{p}\in \Q(\zeta_n)$ if and only if $\d_p\mid n$.  Suppose this is
the case. Let 
  \begin{equation}
    \label{eq:chi}
  \chi: G_n=(\Z/n\Z)^\times \to
  \Gal(\Q(\sqrt{p})/\Q)=\{\pm 1\}, \quad
  \sigma_r(\sqrt{p})=\chi(r)\sqrt{p}  
  \end{equation}  
be the associated quadratic character.  Clearly, $\chi$ factor through
$G_{\d_p}=\Gal(\Q(\zeta_{\d_p})/\Q)$. 

\begin{lemma}\label{lemma:par.2} 
Let $n$ be a
positive integer with $n\not \equiv 2 \pmod 4$ and $q=p^a$ is an odd
power of $p$.  

{\rm (i)} If $\sqrt{p}\not \in \Q(\zeta_n)$, then
  $\sqrt{q}\zeta_n\sim-\sqrt{q}\zeta_n$. 

{\rm (ii)} Suppose that $\sqrt{p} \in \Q(\zeta_n)$, i.e., $n$ is
divisible by $\d_p$.  
Then
\begin{equation}
  \label{eq:par.4}
  \sqrt{q}\zeta_n\sim-\sqrt{q}\zeta_n \iff 4|n \text{
  and \ } \chi(n/2+1)=1. 
\end{equation}
\end{lemma}
\begin{proof}
  (i) As $\sqrt{p}\not \in \Q(\zeta_n)$, there is an element
$\sigma\in \Gal(L/\Q)$ such that $\sigma(\zeta_n)=\zeta_n$ and
$\sigma(\sqrt{p})=-\sqrt{p}$. Then $\sigma(\sqrt{q}\zeta_n)=-
\sqrt{q}\zeta_n$. 

(ii) First, $\sqrt{q}\zeta_n\sim-\sqrt{q}\zeta_n$ if and only if there
  is an 
  element $\sigma_r\in G_n$ such that
  $\sigma_r(\sqrt{q}\zeta_n)=\chi(r) \sqrt{q}
  \zeta_n^r=-\sqrt{q}\zeta_n$. If $\chi(r)=-1$, then 
  $\zeta_n^r=\zeta_n$ and $\sigma_r=1$, which is impossible. 
  If $\chi(r)=1$, then $\zeta_n^r=-\zeta_n$ and hence $4|n$ and $r=n/2+1$ by
  (\ref{eq:even.1}). This concludes our assertion (\ref{eq:par.4}). 
\end{proof}








\begin{prop}\label{prop:par.3} 
Let $n$ and $q$ be as in Lemma~\ref{lemma:par.2}.

{\rm (a)} Suppose that $p=2$. Then 
\begin{equation}
  \label{eq:par.6}
  \sqrt{q}\zeta_n\sim-\sqrt{q}\zeta_n \iff 8\nmid n \text{ or } 16 |n.
\end{equation}

{\rm (b)} Suppose that $p\equiv 1\pmod 4$. Then
\begin{equation}
  \label{eq:par.7}
  \sqrt{q}\zeta_n\sim-\sqrt{q}\zeta_n \iff p\nmid n \text{ or } 
  4p | n.
\end{equation}

{\rm (c)} Suppose that $p\equiv 3\pmod 4$. Then
\begin{equation}
  \label{eq:par.8}
  \sqrt{q}\zeta_n\sim-\sqrt{q}\zeta_n \iff 4p\nmid n 
  \text{ or } 8p | n.
\end{equation}
\end{prop}
\begin{proof}
(a) By Lemma~\ref{lemma:par.2}, we have
$\sqrt{q}\zeta_n\sim-\sqrt{q}\zeta_n$ if and only if either $8\nmid n$ or
$8| n$ and $\chi(n/2+1)=1$.  
Suppose $8|n$.  Note that $\Q(\zeta_8)=\Q(\sqrt{2},\sqrt{-1})$
and $\sqrt{2}=\zeta_8+\zeta_8^{-1}$. It follows that
\begin{equation}
  \label{eq:par.9}
  \chi(r)=
  \begin{cases}
    1 & \text{if $r\equiv 1, 7\pmod 8$}; \\
   -1 & \text{if $r\equiv 3, 5 \pmod 8$}. 
  \end{cases}
\end{equation}
If $8||n$, then $r=n/2+1\equiv 5 \pmod 8$ and $\chi(r)=-1$.
If $16|n$, then $r=n/2+1\equiv 1 \pmod 8$ and $\chi(r)=1$. Thus,
$\sqrt{q}\zeta_n\sim-\sqrt{q}\zeta_n \iff 8\nmid n \text{ or } 16 |
n.$ 

(b) By Lemma~\ref{lemma:par.2}, we have
$\sqrt{q}\zeta_n\sim-\sqrt{q}\zeta_n$ if and only if either $p\nmid n$ or
$4p | n$ and $\chi(n/2+1)=1$. If $4p\,|n$, then $\chi(n/2+1)=1$ since
$n/2+1\equiv 1 \pmod p$. Thus, 
$\sqrt{q}\zeta_n\sim-\sqrt{q}\zeta_n \iff p\nmid n \text{ or } 4p | n$.


(c) By Lemma~\ref{lemma:par.2}, we have
$\sqrt{q}\zeta_n\sim-\sqrt{q}\zeta_n$ if and only if either $4p \nmid n$ or
$4p\, | n$ and $\chi(n/2+1)=1$. 
Suppose that $4p|n$ and write $G_{4p}=G_4 \times G_{p}$. Since
$r=n/2+1\equiv 1 \pmod p$, the image of $\sigma_r$ in $G_p$ is
trivial. In particular, it fixes $\sqrt{-p}\in\Q(\zeta_p)$. 
As $\sqrt{-p}\cdot
\sqrt{-1}=-\sqrt{p}$, one has $\chi(r)=1$ if and only if $r\equiv 1
\pmod 4$. Write $n=4pk$ for some integer $k$. Then $r=2pk+1\equiv 1
\pmod 4$ if and only if $k\equiv 0 \pmod 2$. Therefore, we get
$\sqrt{q}\zeta_n\sim-\sqrt{q}\zeta_n \iff 4p\nmid n 
  \text{ or } 8p | n$. 
\end{proof}

As typical examples, we have 
(a) $\sqrt{2}\zeta_8\not\sim -\sqrt{2}\zeta_8$ and
$\sqrt{2}\zeta_{16}\sim -\sqrt{2}\zeta_{16}$, 
(b) $\sqrt{5}\zeta_{5} \not\sim -\sqrt{5}\zeta_{5}$ and
   $\sqrt{5}\zeta_{20} \sim -\sqrt{5}\zeta_{20}$, and 
(c) $\sqrt{3}\zeta_{12} \not\sim -\sqrt{3}\zeta_{12}$ and 
  $\sqrt{3}\zeta_{24} \sim -\sqrt{3}\zeta_{24}$.

\begin{cor}\label{cor:par.4}
Suppose that $q$ is an odd power of $p$ and $n\not\equiv 2\pmod{4}$. 

{\rm (1)} If $p\equiv 1 \pmod 4$, then 
\[ W_q^{\rm ss}=\{\sqrt{q}\zeta_n\, ;\,  n\not \equiv 2\!\! \pmod 4\,
\}\cup 
    \{-\sqrt{q}\zeta_n\, ;\,  2\nmid n \, \text{and\ } p|n\, \}. 
\] 

{\rm (2)} If $p\equiv 3 \pmod 4$ or $p=2$, then 
\[ W_q^{\rm ss}=\{\sqrt{q}\zeta_n\, ;\,  n\not \equiv 2 \! \! \pmod
4\, \}\cup 
    \{-\sqrt{q}\zeta_n\, ;\,  4p\mid n \, \text{and\ } 
  8p \nmid n \, \}. 
\]      
\end{cor}
\begin{proof}
  (1) By Proposition~\ref{prop:par.3}, $\sqrt{q}\zeta_n\not\sim
      -\sqrt{q}\zeta_n$ if and only if $p|n$ and $4p\nmid n$,
      i.e. $p|n$ and $2\nmid n$. (2) We have $\sqrt{q}\zeta_n\not\sim
      -\sqrt{q}\zeta_n$ if and only if $4p|n$ and $8p\nmid n$. 
\end{proof}
\begin{defn}
\label{sec:par.5}
Let $\d_q$ be the smallest positive integer such that
$\Q(\sqrt{q})\subset \Q(\zeta_{\d_q})$. More specifically, $\d_q=\d_p$
if $q$ is an odd power of $p$, otherwise $\d_q=1$.  We say a positive integer $n$
is {\it critical at $q$} if $\d_q|n$ and $2\d_q\nmid n$. 
\end{defn}
It is clear from the definition that for a fixed $n\in \bbN$, the condition that $n$ is critical
at $q=p^a$ depends only on $p$ and the parity of $a$. 

\begin{prop}\label{prop:par.5}
  Let $n\not \equiv 2 \pmod 4$ be a positive integer and $q=p^a$ a
  power of a prime number $p$. Then 
$\sqrt{q}\zeta_n\sim-\sqrt{q}\zeta_n$ if and only if $n$ is not
  critical at $q$.   
\end{prop}
\begin{proof}
The proposition reduces to either (\ref{eq:even.2}) or
Proposition~\ref{prop:par.3} according to whether $a$ is even or odd
respectively. 
\end{proof}
  

\begin{cor}\label{cor:par.6} 
We have
\[ W_q^{\rm ss}=\{\sqrt{q}\zeta_n\, ;\,  
n\not \equiv 2\!\!\! \pmod 4\, \}\cup
    \{-\sqrt{q}\zeta_n\, ;\, n\not \equiv 2\!\!\! \pmod 4  \,  \text{and
    $n$ is critical at $q$} \, \}. 
\] 
\end{cor}

\section{Dimension of supersingular abelian varieties}
\label{sec:dim}

\subsection{}
\label{sec:dim.1}

Let $q=p^a$ be a power of a prime number $p$, and  $\pi$ a
supersingular Weil $q$-number as in the previous section. 
Replacing $\pi$ by a suitable conjugate, we may assume that $\pi=\pm \sqrt{q}
\zeta_n$ for a positive integer $n$ with $n\not\equiv 2 \pmod 4$.  
Let $X_\pi$ be a simple abelian variety over $\Fq$ 
in the isogeny class corresponding to $\pi$. 
Its endomorphism algebra 
$\calE=\calE_\pi:=\End^0(X_\pi)$ is a central division algebra over
$K:=\Q(\pi)$, unique up to isomorphism depending only on $\pi$ and not
on the choice of $X_\pi$. The field $K$ is either a totally real field or a CM
field \cite[Section~1]{tate:ht}. 
The goal of this section is to determine the dimension $d(\pi)$ of
$X_\pi$. For each $d\in \bbN$, define 
\begin{equation}
  \label{eq:Wd}
  W^{\rm ss}_q(d):=\{\pi\in W^{\rm ss}_q  \mid d(\pi)=d\}.
\end{equation}

According to the
Honda-Tate theory (ibid.), one has 
\[ d(\pi):=\frac{1}{2} [K:\Q] \sqrt{[\calE:K]}=\frac{1}{2}\deg_\Q(\calE). \]
(For a semisimple algebra over a field $F$, its $F$-degree is the degree of any of its
maximal commutative semi-simple $F$-subalgebras.) 
Moreover, the invariants of $\calE$ at a place $v$ of $K$ is given by 
\[ \inv_v(\calE)= 
\begin{cases}
  1/2 & \text{if $v$ is real;} \\
  v(\pi)/v(q) [K_v:\Qp] & \text{if $v|p$;} \\
  0 & \text{otherwise.} 
\end{cases} 
\]
Here $K_v$ is the completion of $K$ at the place $v$. Observe that $d(\pi)=d(-\pi)$. 
As $v(\pi)/v(q)=1/2$ for all $v|p$, every invariant $\inv_v(\calE)$ is a
2-torsion.
It follows from the Albert-Brauer-Hasse-Noether theorem that $\calE$ is either a
quaternion $K$-algebra or the field $K$ itself (henceforth labeled as case (Q) or (F) respectively). 

\subsection{Totally real case}
\label{sec:12}
The case where $K$ is a totally real field is well known. 

(a) If $a$ is even, then $K=\Q$ and $\calE$ is the quaternion algebra over
$\Q$ ramified exactly at $\{p,\infty\}$. One has $\pi=\pm
p^{a/2}$ (two isogeny classes) and $d(\pi)=1$.

(b) If $a$ is odd, then $K=\Q(\sqrt{p})$ 
and $\calE$ is the quaternion algebra over
$K$ ramified exactly at the two real places $\{\infty_1,
\infty_2\}$ of $K$. One has $\pi=q^{1/2}$ (one isogeny class) 
and $d(\pi)=2$.

\subsection{CM case}  
\label{sec:dim.3} Consider the case where $K$ is a CM field,
i.e., $n>2$. 
Put $L:=\Q(\sqrt{q}, \zeta_n)\supseteq K$. As $K$ and $L$ are abelian
extensions of $\Q$, 
the degree $[K_v:\Qp]$ is even for one $v|p$ if and only if it is so for
all $v|p$. Thus, we have the following two possibilities:
\begin{itemize}
\item [(F)] $[K_v:\Q_p]$ is even for all $v|p$.

\item [(Q)] $[K_v:\Q_p]$ is odd for all $v|p$. 
\end{itemize}

As $K$ is CM, Condition (F) holds if and only if all invariants
of $\calE$ vanish. In this case $\calE=K$
and $d(\pi)= [K:\Q]/2$.

\subsection{The case where $a$ is even.} \label{sec:dim.4}
Suppose that $n>2$. One has $K=\Q(\zeta_n)$ and
$[K:\Q]=\varphi(n)$. Thus,
\begin{equation}
  \label{eq:dim.1}
  d(\pi)=
\begin{cases}
  \varphi(n)/2 & \text{if (F) holds}; \\
  \varphi(n) & \text{if (Q) holds.}
\end{cases}
\end{equation}
The ramification index of any ramified prime $p$ in $\Q(\zeta_n)$ is
even,  so if $p\mid n$, then (F) holds. When $p\nmid
n$, Condition (F) holds if and only if the order of $p\in \umod{n}$ is
even. In particular,  if $[K:\Q]$ is a power of 2, then Condition (Q)
holds if 
and only if $K_v=\Qp$, or equivalently $p\equiv 1 \pmod n$. 
We have the following list, which enables to us to list concretely all
$\pi$ with small values of $d(\pi)$. 

\begin{center}
\begin{tabular}{|c|c|c|c|c|c|c|c|c|c|c|c|c|c|c|}  \hline
$n\not\equiv 2\pmod{4}$      & $3$ & $4$ & $5$ & $7$ & $8$ & $9$ & $11$ & $12$ & 
$15$ & $16$ & $20$ & $21$ & $24$ & rest \\ \hline
$d(\pi)$, (Q) holds & $2$ & $2$ & $4$ & $6$ & $4$ & $6$ & $10$ & $4$ & 
$8$  & $8$ & $8$ & $12$ & $8$ & $>8$ \\ \hline
$d(\pi)$, (F) holds  & $1$ & $1$ & $2$ & $3$ & $2$ & $3$ & $5$ & $2$ & 
$4$ & $4$ & $4$ & $6$ & $4$ & $>4$  \\ \hline
\end{tabular} \\ 
\end{center} \ 






\begin{prop}\label{prop:dim.1}
Let $\pi=\pm \sqrt{q} \zeta_n$ be a supersingular Weil $q$-number
with $n\ge 1$ and $n\not \equiv 2\pmod{4}$. Suppose that $q=p^a$ is an even power of $p$. 
  \begin{enumerate}
  \item We have $d(\pi)=1$ if and only if $n=1$, or $n=3, 4$ 
    and $p\not\equiv 1 \pmod n$. 
  \item We have $d(\pi)=2$ if and only if 
    \begin{enumerate}
    \item $n=3,4$ and $p \equiv 1 \pmod n$, or
    \item $n=5, 8, 12$ and $p \not \equiv 1 \pmod n$.
    \end{enumerate}
  \item We have $d(\pi)=3$ if and only if $n=7$ and $p\not\equiv 1,2,4
    \pmod 7$, or $n=9$ and $p\not \equiv 1,4,7 \pmod 9$. 
  \item We have $d(\pi)=4$ if and only if
    \begin{enumerate}
    \item $n=5,8, 12$ and $p \equiv 1 \pmod n$, or
    \item $n=15,16,20,24$ and $p \not \equiv 1 \pmod n$.
    \end{enumerate} 
  \end{enumerate}
\end{prop}


\subsection{The case where $a$ is odd. }
\label{sec:dim.5}
Suppose that $n>1$ and $n\not \equiv 2\pmod{4}$. Put 
\begin{equation}
  \label{eq:def-of-m}
 m:=
\begin{cases}
  n/2 & \text{if $n$ is even,} \\
  n   & \text{if $n$ is odd,}
\end{cases} \quad \text{and}\quad K:=\Q(\sqrt{p}\zeta_n).   
\end{equation}
We have the following inclusions of number fields.
\def\loc{{\rm loc}}
\def\mod{{\rm mod}}
\begin{equation}
  \label{eq:dim.5}
  \xymatrix{
 & L=\Q(\sqrt{p}\,,\zeta_n) \ar@{-}[ld] \ar@{-}[d] \ar@{-}[rd] & \\
\Q(\sqrt{p}\,,\zeta_m) \ar@{-}[dr] & K=\Q(\sqrt{p}\,\zeta_n) \ar@{-}[d]  &
 \Q(\zeta_n) \ar@{-}[ld] \\    
 & E=\Q(\zeta_m) & }
\end{equation}


Note that the prime $p$ is ramified in $K$ with even 
ramification index, and hence Condition (F) always holds. Therefore,
\begin{equation}
  \label{eq:dim.12}
   \calE=K \quad \text{ and } \quad d(\pi)=\frac{1}{2}\, [K:\Q].
\end{equation}








\begin{lemma}\label{lm:dim.2}
  Let $K$ and $E$ be as in (\ref{eq:dim.5}). We have $K=E$ if and only if $n$ is critical at $q$. 
\end{lemma}
\begin{proof}
  Clearly $[K:E]=1$ or $2$. If $\pi\sim -\pi$, then $\pi\mapsto -\pi$
  induces a nontrivial automorphism of $K$ with fixed field $E$. 
  Thus, $\pi \sim -\pi$ if and only if $[K:E]=2$. By
  Proposition~\ref{prop:par.5},  $[K:E]=1$ 
  if and only if $n$ is critical at
  $q$. Note that the lemma also holds when $a$ is even with
  $K=\Q(\sqrt{q}\zeta_n)=\Q(\zeta_n)$. 
\end{proof}


\begin{lemma}\label{lemma:dim.3}
  Suppose that $a$ is odd and $n>1$ with $4\nmid n$. Then
  \begin{equation}
    \label{eq:dim.4}
    d(\pi)=\frac{1}{2}[K:\Q]=
\begin{cases}
 \varphi(n)/2  & \text{if $p\,|n$ and $p\equiv 1\!\!\! \pmod
 4$;} \\  
 \varphi(n) & \text{otherwise.}  
\end{cases}
  \end{equation}
\end{lemma}
\begin{proof}
  Since $n$ is odd one has $E=\Q(\zeta_n)$ and $[E:\Q]=\varphi(n)$.
  We have $\d_q=p$ or $4p$ according as $p\equiv 1 \pmod4$ or not. It is easy to see that $n$ is critical at
  $q$ if and only if $p\equiv 1 \pmod 4$ and $p|n$. The assertion then
  follows from Lemma~\ref{lm:dim.2} and (\ref{eq:dim.12}). 
\end{proof}


\begin{lemma}\label{lemma:dim.4}
  Suppose that $a$ is odd and $n=4k$ with $k\in \bbN$. Then 
\[d(\pi)=\frac{1}{2} [K:\Q]=
 \begin{cases}
    \varphi(n)/4 & \text{if $p\not\equiv 1\pmod{4}$, $4p\mid n$ and
      $8p\nmid n$}; \\
   \varphi(n)/2 & \text{otherwise.}
 \end{cases}\]
\end{lemma}
\begin{proof}
  Since $4|n$ we have $[E:\Q]=\varphi(n)/2$. 
  By Lemma~\ref{lm:dim.2} 
  we have
  $[K:\Q]=\delta_n \varphi(n)/2$, where $\delta_n=1$ or $2$
  depending on  
  $n$ is critical at $q$ or not. The lemma follows once we note that $n=4k$ is
  never critical when $p\equiv 1\pmod{4}$.  
\end{proof}

The following is a list of $d(\pi)$ for $\pi=\sqrt{q}\zeta_n$ with
$4\nmid n$ and $4|n$, respectively. 
The symbol ($*$) denotes the primes satisfying the 
conditions $p\, |n$ and $p\equiv 1\, (4)$,
and ($**$) denotes the primes satisfying the three conditions $p\not\equiv
1\pmod{4}$, $4p\mid n$ and $8p\nmid n$.  For the sake of completeness,
the case $n=1$ is included and also marked to make a distinction.


\begin{center}
\begin{tabular}{|c|c|c|c|c|c|c|c|c|c|}  \hline
$n$ odd           & $1^\natural$ & $3$ & $5$ & $7$ & $9$ & $11$ & $13$ & $15$ 
 & rest \\ \hline 
$\varphi (n)$ & $1$ & $2$ & $4$ & $6$ & $6$ & $10$ & $12$ & $8$ 
& $>8$ \\ \hline
($*$) & $\emptyset$ & $\emptyset$ & $p=5$ 
      & $\emptyset$ & $\emptyset$ & $\emptyset$ & $p=13$ & $p=5$ &  
\\ \hline
$d(\pi) $ & $2$ &  $2$ & $2\ (p=5)$ & $6$ & $6$ & $10$  & $6\ (p=13)$ &
$4 \ (p=5)$ & $>4$ \\  
         &  &  & $4\ (p\neq 5)$ &   &  & & $12 \ (p\neq 13)$
         & $8 \ (p\neq5)$ & \\ \hline 
\end{tabular} \\ 
\end{center} 
 

\begin{center}
\begin{tabular}{|c|c|c|c|c|c|c|c|}  \hline
$n=4k$ & $4$ & $8$ & $12$  & $16$ & $20$ & $24$ & $28$ \\ \hline
$\varphi (n)$ & $2$ & $4$ & $4$ & $8$ & $8$ & $8$ & $12$ \\ \hline
($**$) &  $\emptyset$ & $2$     & $3$ &  $\emptyset$ & $\emptyset$ 
  & $2$ & $7$ \\ \hline 
$d(\pi)$ &  $1$ & $1\ (p=2)$     & $1\ (p=3)$ & $4$  & $4$ 
& $2\ (p=2)$ & $3\ (p=7)$ \\ 
  &  & $2\ (p\neq2)$  & $2\ (p\neq3)$ &   & & $ 4\ (p\neq  2)$ 
   & $ 6\ (p\neq 7)$ \\ \hline \hline
 $n=4k$ &       $32$ & $36$ & $40$ & $44$ & $48$ & $56$ & $60$    \\ \hline
$\varphi(n)$ & $16$ & $12$ & $16$ & 20  & $16$ & $24$ & $16$  \\ \hline
($**$)  &    $\emptyset$ &   $p=3$ & $p=2$  & $p=11$ & $\emptyset$ & $p=2$ & $p=3$ 
   \\ \hline 
$d(\pi)$ &  $8$ & $3\ (p=3)$  & $4\ (p=2)$ & $5\ (p=11)$ & $8$ & 
$6\ (p=2)$ & $4\ (p=3)$  \\ 
  & & $6\ (p\neq 3)$  & $8\ (p\neq2)$ & $10\ (p\neq 11)$ &  & 
$12\ (p\neq 2)$ & $ 8\ (p\neq 3)$  \\ \hline 
\end{tabular} \\ 
\end{center} \

It is easy to see that when $4|n$ and either $n=52$ or $n>60$, the value
$\varphi(n)>16$ and hence $d(\sqrt{q}\zeta_n)>4$. 

\begin{prop}\label{prop:dim.5}
  Suppose that $q=p^a$ is an odd power of $p$. 
  
\begin{enumerate}
\item $W^{\rm ss}_q(1)$ consists of 
\[ \sqrt{q} \zeta_4,\ \pm \sqrt{q} \zeta_8\ (p=2), \ 
\pm \sqrt{q}\zeta_{12} \ (p=3). \] 
\item $W^{\rm ss}_q(2)$ consists of 
\[ \sqrt{q},\  \sqrt{q} \zeta_3,\ \pm \sqrt{q} \zeta_5\ (p=5),
\  \sqrt{q}\zeta_8\ (p\neq 2),\  \sqrt{q}\zeta_{12} \ (p\neq3),
\ \pm\sqrt{q}\zeta_{24}\ (p=2). \]
\item $W^{\rm ss}_q(3)$ consists of $\pm \sqrt{q}\zeta_{28}$ if $p=7$, or  $\pm\sqrt{q}\zeta_{36}$ if $p=3$.

\item $W^{\rm ss}_q(4)$ consists of 
\[ \sqrt{q} \zeta_{5}\ (p\neq 5), 
\ \pm \sqrt{q} \zeta_{15}\ (p=5), \ \sqrt{q} \zeta_{16}, \]
\[ \sqrt{q} \zeta_{20},
\  \sqrt{q} \zeta_{24}\ (p\neq 2), 
\ \pm \sqrt{q} \zeta_{40}\ (p=2),\ \pm \sqrt{q} \zeta_{60}\  (p=3). \]
\end{enumerate}
\end{prop}





\section{Supersingular elliptic curves over finite fields}
\label{sec:curve}

\def\Isog{\mathrm{Isog}}

\subsection{Isogeny classes over finite fields}
\label{sec:curve.1}

Let ${\mathcal Isog}_q$ denote the set of isogeny classes of abelian
varieties over $\Fq$, where $q=p^a$ is a power of the prime number
$p$. Let $\Z W_q$ be the free abelian group (written
multiplicatively) generated by the set $W_q$ of conjugacy classes of 
Weil $q$-numbers. A nontrivial
element $\pi\in \Z W_q$ can be put in the form $\pi_1^{m_1}\times
\dots \times \pi_r^{m_r}$ for some $r\in \bbN$,
where each $\pi_i\in W_q$,  $\pi_i\not\sim \pi_j$ if $i\neq j$, and
$m_i\neq 0$ for all $1\leq i\leq r$. Such an element 
is called a {\it multiple Weil $q$-number} if $m_i> 0$ for all $i$,
and the set of all these elements is denoted by $MW_q$. 
Put $X_\pi:=\prod_i X_{\pi_i}^{m_i}$, where $X_{\pi_i}$ is the simple abelian
variety (up to isogeny) over $\Fq$ corresponding to $\pi_i$. 
The Honda-Tate theorem naturally extends to a bijection $MW_q \simeq
{\mathcal Isog}_q$ which sends each $\pi\in MW_q$ to the isogeny class 
$[X_\pi]\in {\mathcal Isog}_q$ of $X_\pi$.   

For each $\pi\in MW_q$, we define its {\it dimension} as 
\[ d(\pi):=\dim X_\pi= \sum_{i=1}^r m_i d(\pi_i). \]
Let $\Isog(\pi)=\Isog(X_\pi)$ 
denote the set of $\Fq$-isomorphism classes of
abelian varieties 
isogenous to $X_\pi$ over $\Fq$, and denote $H(\pi):=|\Isog(\pi)|$. 
Let $MW^{\rm ss}_q\subset MW_q$ be the subset of supersingular multiple
Weil $q$-numbers, i.e.~those $\pi\in MW_q$ whose corresponding 
abelian varieties $X_\pi$ are supersingular.
For any integer $d\ge 1$, let $MW_q(d)$ 
(resp. $MW^{\rm ss}_q(d)$) denote
the subset consisting of all elements $\pi$ in $MW_q$ 
(resp. in $MW^{\rm ss}_q$) of dimension $d$. Let $S_d(\F_q)$ 
(resp. ${\rm Sp}_d(\F_q)$) be the set
of isomorphism classes of 
$d$-dimensional supersingular (resp.~superspecial) abelian varieties
over $\Fq$. When $\pi\in MW^{\rm ss}_q$, we let ${\rm Sp}(\pi)\subset
  \Isog(\pi)$ be the subset consisting of superspecial isomorphism classes and denote
  $H_{sp}(\pi):=|{\rm Sp}(\pi)|$. Thus,
\begin{equation}
    \label{eq:sp.1}
    |S_d(\F_q)|=\sum_{\pi \in MW^{\rm ss}_q(d)} H(\pi), \quad  
    |{\rm Sp}_d(\F_q)|=\sum_{\pi \in MW^{\rm ss}_q(d)} H_{\rm sp}(\pi).
\end{equation}

\subsection{Supersingular elliptic curves}
\label{sec:curve.2} 
We compute the number $|S_1(\Fq)|$ of isomorphism classes of
supersingular elliptic curves over $\Fq$, where $q=p^a$ as before. 
The method is based almost entirely on the results of Waterhouse
\cite{waterhouse:thesis}, except certain details need to be cleared
up (compare with \cite[Theorems 4.5]{waterhouse:thesis}).  

\begin{thm}\label{curve.1}
  Let $\pi$ be the Frobenius
  endomorphism  of an elliptic curve $E_0$ over $\Fq$,  and $K:=\Q(\pi)$. 
  Assume that $\pi\not\in\Q$ so that
  $K$ is an imaginary quadratic field. Equivalently, the central $K$-algebra
  $\End^0(E_0)$ of the elliptic curve $E_0$ is assumed to be
  commutative and thus necessarily an imaginary quadratic field. 

  \begin{enumerate}
  \item Any endomorphism ring $R=\End(E)$ of an elliptic curve $E$ in
    the isogeny class $[E_0]$ of $E_0$ contains $\pi$ and is maximal at
    $p$, that is, $R\otimes \Z_p$ is the maximal order in $K\otimes
    \Qp$. Conversely, any
    order $R$ of $K$ satisfying these two properties occurs as an
    endomorphism ring of an elliptic curve in this isogeny class. 
  \item Suppose that $R\subset K$ is a quadratic order as in (1). 
    Then the Picard
    group $\Pic(R)$ of $R$ acts freely on the set $[E_0]_R\subset [E_0]$ 
    of isomorphism
    classes of elliptic curves in $[E_0]$ with endomorphism ring
    $R$. Moreover, the
    number $N$ of orbits is $2$ if $p$ is inert in $K$ and $a$
    is even, and $N=1$ otherwise.   
  \end{enumerate}
\end{thm}
\begin{proof} 
  Statement (1) is \cite[Theorem 4.2]{waterhouse:thesis}.
  We give a proof of the second part of Statement (2) 
  since it differs from
  \cite[Theorem 4.5]{waterhouse:thesis} in some cases. 
  We assert that the statement of \cite[Theorem
  5.1]{waterhouse:thesis} for principal abelian varieties is
  directly applicable to this situation. Namely, the number of orbits here
  is also given by $N=\prod_{v|p} N_v$,
  where $v$ runs through the set of all places of $K$ over $p$, 
  and each $N_v$ is the number described as follows. 
  Let $e_v$ and $f_v$ be 
  the ramification index 
  and residue degree of $v$, respectively, and set $g_v=\gcd(f_v, a)$
  and $m_v:=g_v \ord_v(\pi)/a$. Note that $m_v$ is an integer since $\End^0(E_0)$
  is commutative and thus $f_v \ord_v(\pi)/a\in \bbN$. Then $N_v$ is the number of all 
  $g_v$-tuples $(n_1,\dots, n_{g_v})$ of
  integers satisfying $0\le n_j\le e_v$ and $\sum_{j=1}^{g_v} n_j=m_v$.

  In the present situation $\End^0(E_0)=K$ is commutative and $R$ is maximal at $p$. 
  As in the proof of \cite[Theorem
  5.1]{waterhouse:thesis}, to find the
  the number of orbits for the action of $\Pic(R)$ on $[E_0]_R$, one
  needs to classify the Tate-modules $T_\ell E$ at all primes $\ell
  \neq p$ and  the \dieu modules at the prime $p$ of $E\in[E_0]_R$.  The number
  of orbits is then the product of the number of isomorphism classes
  of the above modules at each prime. 

The Tate-module $T_\ell E$ of each $E\in [E_0]_R$ at a prime $\ell
  \neq p$ is naturally an $R_\ell$-module with
  $R_\ell=R\otimes_\Z\Z_\ell$. 
Since $R[1/p]$ is a
  quadratic order, any fractional $R[1/p]$-ideal $I$ whose order
  ring equals $R[1/p]$ must be locally free over $R[1/p]$.  Particularly, there is only one isomorphism class
  of the prime-to-$p$ Tate modules of $E$ for  all
  $E\in [E_0]_R$.   Thus, $N$ is equal to the
   number of isomorphism classes of \dieu modules occurring in the
  isogeny class 
  $[E_0]$, which is equal to $\prod_v N_v$ as given in the proof of 
  \cite[Theorem 5.1]{waterhouse:thesis}. 

  

  Now it is easy to compute the number $N$ of orbits. 
  Notice $N_v\neq 1$ only when $g_v>1$. For our
  case with $[K:\Q]=2$ this occurs only when $p$ is inert in $K$ and
      $a$ is even. 
      In this case there is only one place $v$ over $p$, $g_v=2$ and
      $e_v=1$. Then $N=N_v$ is the number of pairs $(n_1,n_2)$ with
      $0\le n_1, n_2\le 1$ and $n_1+n_2=1$, which is $2$. 
\end{proof}

\begin{remark}
  In \cite[Theorem 5.1]{waterhouse:thesis} the assumption that
  the endomorphism ring $R=\End(A)$ is the maximal order  can be replaced
  by the weaker assumption that $R$ is
  both Gorenstein and maximal at $p$. Indeed, any proper
  $R$-lattice of rank one over a Gorenstein order $R$ is locally free 
  \cite[Theorem 37.16 p.~789]{curtis-reiner:1}, so the same proof of
  \cite[Theorem 5.1]{waterhouse:thesis} applies. 
\end{remark}

\begin{rem}\label{rem:2-orbits-case}
Suppose that $a$ is even and $p$ is inert in the imaginary quadratic 
field $K=\Q(\pi)$ so that
$N=2$. By the classification of Waterhouse
(\cite[Lemma,~p.537]{waterhouse:thesis}, see also 
Proposition~\ref{prop:dim.1}), this occurs only for  supersingular
Weil $q$-numbers $\pi$ where 
\begin{equation}\label{eq:1-dim-weil-number-2-orbits}
\pi\sim \pm p^{a/2}\zeta_3,\, p\equiv 2\pmod{3} \quad \text{ or } \quad 
 \pi \sim p^{a/2}\zeta_4, \, p\equiv 3\pmod{4}. 
\end{equation}
Then by part (1) of Theorem~\ref{sec:curve.1}, $\End(E)=O_K$ for any
elliptic curve $E$ in the isogeny class corresponding to $\pi$.  Since $h(O_K)=1$, part (2) of
Theorem~\ref{sec:curve.1} implies that a complete set of
representatives of $\Sp(\pi)$ consists a pair of
elliptic curves of the form $\{E, E^{(p)}\}$, where
$E^{(p)}:=E\otimes_{\F_q, \sigma_p} \F_q$, and $\sigma_p\in \Gal(\F_q/\F_p)$ is the Frobenius automorphism of
$\F_q/\F_p$.  These two elliptic curves are distinguished by the
actions of $O_K$ on the respective 1-dimensional Lie-algebras
$\Lie(E)$ and $\Lie(E^{(p)})$ over
$\F_q$, which are given by distinct embeddings
$O_K/(p)\simeq \F_{p^2} \hookrightarrow \F_q$. This establishes a
natural bijection
$\Sp(\pi)\simeq \Hom(O_K/(p), \F_q)$ for every $\pi$ in (\ref{eq:1-dim-weil-number-2-orbits}). 
\end{rem}

%




We return to the calculation of $|\Sp_1(\Fq)|$ by the counting
method. The isogeny classes of supersingular
elliptic curves over $\F_q$ are completely listed by the following Weil numbers
\begin{equation}
  \label{eq:curve.2}
  \begin{split}
    W_q^{\rm ss}(1)&=\{\sqrt{q} \zeta_4,\ \pm \sqrt{q} \zeta_8\ (p=2),
    \ \pm \sqrt{q} \zeta_{12}\ (p=3)\, \}, \quad \text{for $a$ odd}; \\
    W_q^{\rm ss}(1)&=\{\pm \sqrt{q},\ 
    \pm \sqrt{q} \zeta_3\ (p\not\equiv 1\, (3)),\  
    \sqrt{q} \zeta_{4}\ (p\not\equiv 1 \,(4))\, \},  
     \quad \text{for $a$ even.} \\
  \end{split}
\end{equation}
For each Weil $q$-number $\pi\in W_q^{\rm ss}(1)$, let $R_0$ be the
smallest quadratic order in $K=\Q(\pi)$ which contains 
$\pi$ and is maximal at $p$.
It is easy to see that $R_0$ is the maximal order except when 
$\pi=\sqrt{q} \zeta_4$, $p\equiv 3\pmod{4}$ and $a$ is odd. In the 
latter case 
$R_0=\Z[\sqrt{-p}]$ and we have by Theorem~\ref{curve.1} that
\begin{equation}
  \label{eq:4odd}
  H(\sqrt{q} \zeta_4)=
  \begin{cases}
    h(O_K) & \text{for $p=2$ or $p\equiv 1 \pmod 4$;} \\
    h(R_0)+h(O_K) & \text{for $p\equiv 3 \pmod 4$.}
  \end{cases}
\end{equation}
For the other cases the order $R_0$ is maximal and we have 
\begin{equation}
  \label{eq:curve.4}
  H(\pi)=N\cdot h(O_K)
\end{equation}  
where $N=2$ if $p$ is inert in $K$ and $a$ is even, and $N=1$ otherwise.
Recall that for a square free $m\in \Z$, the class
number of $\Q(\sqrt{m})$ is denoted by $h(\sqrt{m})$. 

Suppose first that $a$ is odd. For $p=2$, we have
\begin{equation}
  \label{eq:oddp=2}
  |\Sp_1(\F_q)|=H(\sqrt{q}\zeta_4)+2
   H(\sqrt{q}\zeta_8)=h(\sqrt{-2})+2h(\sqrt{-1})=3. 
\end{equation}

For $p=3$, we have
\begin{equation}
  \label{eq:oddp=3}
  \begin{split}
  |\Sp_1(\F_q)|&=H(\sqrt{q}\zeta_4)+2
   H(\sqrt{q}\zeta_{12}) \\
  &=h(\Z[\sqrt{-3}])+h(\sqrt{-3})+2h(\sqrt{-3})=4.  
  \end{split}   
\end{equation}
For $p>3$, we have (see \cite[Theorem 1.1]{yu:sp_prime})
\begin{equation}
  \label{eq:oddp>3}
  \begin{split}
  |\Sp_1(\F_q)|&=H(\sqrt{q}\zeta_4) \\
  &=
  \begin{cases}
    h(\sqrt{-p}) & \text{for $p\equiv 1 \pmod 4$;}\\
    2h(\sqrt{-p}) & \text{for $p\equiv 7 \pmod 8$ ($2$ splits in
    $\Q(\sqrt{-p})$) ;}\\
    4h(\sqrt{-p}) & \text{for $p\equiv 3 \pmod 8$ ($2$ is inert in
    $\Q(\sqrt{-p})$).}
  \end{cases}    
  \end{split}
\end{equation}
Since $\left (\frac{2}{p}\right )=1$ for $p\equiv 1,7 \pmod 8$ and
$\left (\frac{2}{p}\right )=-1$ for $p\equiv 3, 5 \pmod 8$, we can
rewrite (\ref{eq:oddp>3}) as
\begin{equation}
  \label{eq:odd_p>3}
  |\Sp_1(\F_q)|=
  \begin{cases}
    h(\sqrt{-p}) & \text{for $p\equiv 1 \pmod 4$;}\\
    \left(3-\left (\frac{2}{p}\right )\right )h(\sqrt{-p}) & 
    \text{for $p\equiv 3 \pmod 4$}.
  \end{cases}    
\end{equation}

Suppose now that $a$ is even. 
By (\ref{eq:curve.2}), we have
\begin{equation}
  \label{eq:curve.9}
  |\Sp_1(\F_q)|=2H(\sqrt{q})+2 \delta_3(p) H(\sqrt{q}\zeta_3)
  +\delta_4(p) H(\sqrt{q}\zeta_4),  
\end{equation}
where $\delta_m(p)=1,0$ according as $p\not\equiv 1 \pmod m$ or not
for $m=3,4$.
It is well known that $H(\sqrt{q})$ is equal to the class number
$h(B_{p,\infty})$ of the quaternion $\Q$-algebra
$B_{p, \infty}$
ramified only at $p$ and $\infty$. Thus, 
\begin{equation}
  \label{eq:curve.10}
  H(\sqrt{q})=\frac{p-1}{12}+\frac{1}{3}\left
  (1-\left(\frac{-3}{p}\right )\right )+\frac{1}{4}\left
  (1-\left(\frac{-4}{p}\right )\right ).
\end{equation}
By Theorem~\ref{curve.1}, we have
\begin{equation}
  \label{eq:even3}
  \delta_3(p)H(\sqrt{q}\zeta_3)=
  \begin{cases}
    1 & \text{for $p=3$;}\\
    2 & \text{for $p\equiv 2 \pmod 3$;} \\
    0 & \text{for $p\equiv 1 \pmod 3$;}
  \end{cases}
\end{equation}
and get $\delta_3(p)H(\sqrt{q}\zeta_3)=\left
  (1-\left(\frac{-3}{p}\right )\right )$. 
  Similarly we have $\delta_4(p)H(\sqrt{q}\zeta_4)=\left
  (1-\left(\frac{-4}{p}\right )\right )$. Using (\ref{eq:curve.9}) and
  (\ref{eq:curve.10}) we get  
\begin{equation}
    \label{eq:curve.12}
    \begin{split}
  |\Sp_1(\F_q)|&=\frac{p-1}{6}+\frac{2}{3}\left
  (1-\left(\frac{-3}{p}\right )\right )+\frac{1}{2}\left
  (1-\left(\frac{-4}{p}\right )\right )\\
  &\quad \quad +2\left
  (1-\left(\frac{-3}{p}\right )\right )+\left
  (1-\left(\frac{-4}{p}\right )\right ) \\
  &= \frac{p-1}{6}+\frac{8}{3}\left
  (1-\left(\frac{-3}{p}\right )\right )+\frac{3}{2}\left
  (1-\left(\frac{-4}{p}\right )\right ).      
    \end{split}    
\end{equation}
From (\ref{eq:oddp=2}), (\ref{eq:oddp=3}), (\ref{eq:odd_p>3}) and 
(\ref{eq:curve.12}), we obtain an explicit formula for the 
number $|\Sp_1(\F_q)|$ of supersingular elliptic curves over $\Fq$.

\begin{thm}\label{curve.3}
  Suppose $q=p^a$ is a power of the prime number $p$.
\begin{enumerate}
  \item If $a$ is odd, then 
\begin{equation}
    \label{eq:curve.13}
  |\Sp_1(\F_q)|=
  \begin{cases}
    3, 4 & \text{for $p=2,3$, respectively;} \\
    h(\sqrt{-p}) & \text{for $p\equiv 1 \pmod 4$;}\\
    \left(3-\left (\frac{2}{p}\right )\right )h(\sqrt{-p}) & 
    \text{for $p\equiv 3 \pmod 4$ and $p>3$}.    
  \end{cases}
  \end{equation}
  \item If $a$ is even, then 
\begin{equation}
    \label{eq:curve.14}
    \begin{split}
  |\Sp_1(\F_q)| = \frac{p-1}{6}+\frac{8}{3}\left
  (1-\left(\frac{-3}{p}\right )\right )+\frac{3}{2}\left
  (1-\left(\frac{-4}{p}\right )\right ).      
    \end{split}    
\end{equation}
\end{enumerate}
\end{thm}


\begin{remark}\label{curve.4}
  From the formulas above we observe a phenomenon that the number
  $|\Sp_1(\Fq)|$ depends only on the parity of the exponent $a$ of
  $q=p^a$. We have already seen in Section 2 that the classification
  of supersingular isogeny classes depends only on the parity of
  $a$. More explicitly, if the exponents $a$ and $a'$ of $q$ and $q'$ respectively have
  the same parity, then a bijective correspondence between
  supersingular isogeny classes over
  $\F_q$ and those over $\F_{q'}$ can be given by matching
  $\pi\in W_q^{\rm ss}(1)$ with $\pi'=(-p)^{(a'-a)/2}\pi$ (see
  Remark~\ref{rem:parity-isogeny-class}).  The parity phenomenon of
  $|\Sp_1(\Fq)|$ arises because there is a bijection
  $\Sp(\pi)\simeq \Sp(\pi')$ for all pairs  $(\pi, \pi')$ as above. Indeed, if $\pi$
  and $\pi'$ are of the form in
  (\ref{eq:1-dim-weil-number-2-orbits}), then a canonical bijection
  $\Sp(\pi)\simeq \Sp(\pi')$ is given by identifying both with
  $\Hom(O_K/(p), \F_q)$ as in Remark~\ref{rem:2-orbits-case}.  For the
  remaining cases, first suppose that $K=\Q(\pi)=\Q(\pi')$ is
  imaginary quadratic.  Then the endomorphism rings occurring for both
  isogeny classes are the same by Theorem~\ref{curve.1}. We partition
  $\Sp(\pi)$ into $\coprod_{R} \Sp(\pi, R)$, where $R$ runs over all
  possible endomorphism rings, and $\Sp(\pi, R)\subseteq
  \Sp(\pi)$ consists of those members with endomorphism ring $R$.
  Every $\Sp(\pi, R)$ is  
a principal homogeneous
  space of $\Pic(R)$. Thus a $\Pic(R)$-equivariant bijection between
  $\Sp(\pi, R)$ and $\Sp(\pi', R)$  is established whenever a base point is
  chosen respectively in each of them. Lastly,  suppose that
  $\Q(\pi)=\Q(\pi')=\Q$. 
  Then $\pi^{a'}=(\pi')^a=p^{aa'/2}$. So we have canonical bijections
  $\Sp(\pi)\simeq \Sp(\pi^{a'})\simeq \Sp(\pi')$ by
  extending both base fields to $\F_{p^{aa'}}$ (\cite[
  Remark,~p.~542]{waterhouse:thesis}). Equivalently, the bijection
  $\Sp(\pi)\simeq 
  \Sp(\pi')$ can be obtained by matching the $j$-invariants. 
\end{remark}



 
\section{Superspecial abelian surfaces over $\Fp$}
\label{sec:sp}

In this section we assume that the ground field is the prime 
field $\Fp$; abelian varieties and their morphisms are all 
defined over $\Fp$ unless otherwise stated.

\subsection{Supersingular abelian varieties over $\Fp$}
\label{sec:sp.1}

We describe a result which allows us to count supersingular 
and superspecial abelian
varieties over $\Fp$, based on a result of 
Waterhouse~\cite[Theorem 6.1 (3)]{waterhouse:thesis} (also see an
extension~\cite[Theorem 3.1]{yu:sp_prime} to non-simple isogenies).

  
Let $X_0$ be a fixed supersingular abelian variety over $\Fp$ and 
let $\pi=\pi_1^{m_1}\times \dots \times \pi_r^{m_r}$ be a multiple Weil
$p$-number corresponding to the isogeny class $[X_0]$. One has
$X_0\sim \prod_{i=1}^r X_i^{m_i}$, where each 
$X_i$ is a simple abelian variety with Frobenius endomorphism
$\pi_i$. 
The endomorphism algebra $\calE=\End^0(X_0)$ of $X_0$ is
equal to 
$\prod_{i=1}^r \Mat_{m_i}(\End^0(X_i))$. Let $\pi_0\in \End(X_0)$ be 
the Frobenius endomorphism. The $\Q$-subalgebra $K=\Q(\pi_0)\subset
\calE$ generated by $\pi_0$ is semi-simple and 
coincides with the center of $\calE$. One has
$K=\prod_i K_i$ and $\pi_0=(\pi_1,\dots, \pi_r)$, 
where $K_i=\Q(\pi_i)$. Let $\calR:=\Z[\pi_0,p \pi_0^{-1}]\subset
K$ and $\calR_{sp}:=\calR[\pi_0^2/p]\subset K$. 
Clearly $\pi_0^2/p$ is an integral element of finite multiplicative order, and
$p/\pi_0=\pi_0\cdot (\pi_0^2/p)^{-1}$, so $\calR_{sp}=\zz[\pi_0,
\pi_0^2/p]\subseteq O_K$, where $O_K=\prod_i
O_{K_i}$ is  the maximal order $K$.  Observe that the Tate module $T_\ell(X_0)$ (for any prime
$\ell\neq p$), as a $\Z_\ell[\Gal(\Fpbar/\Fp)]$-module, 
is nothing but an $\calR_\ell$-module, and the (covariant) \dieu
module $M(X_0)$ is simply an $\calR_p$-module, where
$\calR_\ell=\calR\otimes \Z_\ell$ and $\calR_p=\calR\otimes \Z_p$. 


 
\begin{thm}\label{prop:sp.1} Let $\pi=\pi_1^{m_1}\times \dots
  \pi_r^{m_r}$, and $K$, $\calR$ and $\calR_{sp}$ be as above. Assume that
  $K$ has no real place, that is, none of $\pi_i$ is conjugate to
  $\sqrt{p}$, and set $V:=\prod_{i=1}^r K_i^{m_i}$.   
\begin{enumerate}
  \item There is a natural 
     bijection between the set ${\rm Isog}(\pi)$ and the
     set of isomorphism classes of $\calR$-lattices in $V$. 
  \item Under the above map the subset ${\rm Sp}(\pi)$ is in bijection 
     with set of isomorphism classes of $\calR_{sp}$-lattices in $V$.
\end{enumerate}
\end{thm}
\begin{proof}
Set $\Lambda:=\prod_{i=1}^r O_{K_i}^{m_i}\subset V$. We view $V$ and $\Lambda$ as
a $K$-module and an $\calR$-lattice, respectively. One chooses an
identification $V\otimes_\Q \Q_\ell = T_\ell(X_0)\otimes \Q_\ell$ for
primes $\ell\neq p$ and $V\otimes_\Q \Q_p=M(X_0)\otimes \Q_p$ such
that $\Lambda_\ell=T_\ell(X_0)$ for almost all primes $\ell$. 
Under this identification, any $\calR$-lattice $\Lambda'$ 
in $V$ gives rise to a unique
quasi-isogeny $\varphi: X\to X_0$ such that
$\varphi_*(T_\ell(X))=\Lambda'\otimes \Z_\ell$ for $\ell\neq p$ and  
$\varphi_*(M(X))=\Lambda'\otimes \Z_p$.
Two lattices $\Lambda_1$ and
$\Lambda_2$ are isomorphic as $\calR$-modules 
if and only if there is an element $g\in
\GL_K(V)$ such that $\Lambda_2=g \Lambda_1$. Two quasi-isogenies are
isomorphic if and only if they differ by an element in
$\calE^\times$. Our assumption ensures
that $\GL_K(V)\simeq \calE^\times$. 
Then the above correspondence induces the desired bijection 
(also see \cite[Theorem 3.1]{yu:sp_prime} for a detailed proof). 


Note that
the abelian variety $X$ in $[X_0]$ as above 
is superspecial if and only if $\pi_0^2 M(X)=pM(X)$, or equivalently,
$M(X)$ is a $(\calR_{sp})_p$-lattice in $M(X_0)\otimes \Q_p$. That is,
$X$ is superspecial if and only if the corresponding $\calR$-module is
$\calR_{sp}$-stable. The statement (2) then follows from (1). 
\end{proof}

\begin{rem}\label{rem:critical-max-order}
  Let $\pi=\pi_1^{e_1}$ be a multiple supersingular Weil-$p$ number with $\pi_1=\pm
  \sqrt{p}\zeta_n$ and $n$ critical
  at $p$. Then by Lemma~\ref{lm:dim.2},  $K=\Q(\pi_1)=\Q(\zeta_m)$ and
  $O_K=\Z[\zeta_m]$, where $m$ is define in
  (\ref{eq:def-of-m}).  Since $\calR_{sp}=\calR[\pi_1^2/p]\ni
  \zeta_m$, it follows that $\calR_{sp}$ coincides with the maximal order
  $O_K$ in this case. 

\end{rem}
\subsection{Proof of the main theorem}
\label{sec:sp.2}

By Section~\ref{sec:dim}, we list the sets $W^{\rm ss}_p(1)$ and
$W^{\rm ss}_p(2)$ of supersingular Weil $p$-numbers of dimension $1$ or
$2$ as follows.
\begin{align} \label{simple_dimension_one}
&W^{\rm ss}_2(1) = \{ \sqrt{2}\zeta_4, \pm \sqrt{2}\zeta_8\}, \notag \\ 
&W^{\rm ss}_3(1) = \{ \sqrt{3}\zeta_4, \pm \sqrt{3}\zeta_{12} \}, \\ 
&W^{\rm ss}_p(1) = \{ \sqrt{p}\zeta_4 \}, \quad p \geq 5; \notag
\end{align}
and
\begin{align}\label{simple_dimension_two}
&W^{\rm ss}_2(2) = \{\sqrt{2}, \sqrt{2}\zeta_3, \sqrt{2}\zeta_{12},
\pm \sqrt{2}\zeta_{24} \}, \notag\\ 
&W^{\rm ss}_3(2) = \{\sqrt{3}, \sqrt{3}\zeta_3, \sqrt{3}\zeta_8 \},\\
&W^{\rm ss}_5(2) = \{\sqrt{5}, \sqrt{5}\zeta_3,
\sqrt{5}\zeta_8, \sqrt{5}\zeta_{12}, \pm \sqrt{5}\zeta_5\}, \notag \\ 
&W^{\rm ss}_p(2) = \{\sqrt{p}, \sqrt{p}\zeta_3, \sqrt{p}\zeta_8,
\sqrt{p}\zeta_{12}\}, \quad p \geq 7. \notag 
\end{align}

Consider the case $\pi\in W^{\rm ss}_p(2)$ or $\pi=\pi_1\times \pi_2$ with
$\pi_1,\pi_2\in W^{\rm ss}_p(1)$. By
(\ref{eq:sp.1}) we have
\begin{equation}
  \label{eq:sp.4}
  |{\rm Sp}_2(\Fp)|=\sum_{\pi\in W^{\rm ss}_p(2)} H_{sp}(\pi)+
    \sum_{\pi_1, \pi_2\in W^{\rm ss}_p(1)} H_{sp}(\pi_1\times \pi_2).
\end{equation}
The number $H_{sp}(\sqrt{p})=H(\sqrt{p})$ has 
been calculated in \cite{xue-yang-yu:ECNF}, so
this case will be excluded from our discussion. We refer to
\cite[Section 37]{curtis-reiner:1} for the definition of a Bass
order. Note that when $\pi=\pi_1\times \pi_1$,  $\calR_{sp}$ is an order in the quadratic field
$\Q(\pi_1)$, and such orders are well known to be Bass.
It will be shown in Section~\ref{subsec:order-Rsp-bass} that $\calR_{sp}$
is a Bass order for all $\pi$ considered 
(i.e. $\pi\in MW^{\rm ss}_p(2)$).   
Thus, when the $K$-module
$V$ is free of rank one 
(i.e. in the case where $\pi\neq \pi_1\times \pi_1$), 
Theorem~\ref{prop:sp.1} gives 
\begin{equation}
  \label{eq:sp.5}
  H_{sp}(\pi)=\sum_{\calR_{sp}\subset B\subset O_K} h(B). 
\end{equation}
In the case when $V$ is free of higher rank (in fact, rank $2$
when $\pi=\pi_1\times \pi_1$) , one can use the results
of Borevich and Faddeev on lattices over orders of cyclic index to
compute $H_{sp}(\pi)$ (cf. \cite[Section 37, p.~789]{curtis-reiner:1}).  

In the following, the notation $B_{\pi,j}$ (or $B_j$ for short) with $j\in \bbN$, 
will stand for an order $B$ of $K$ with 
$\calR_{sp}\subset B\subset O_K$ and $[O_K:B]=j$.
The dependence of $K$, $\calR_{sp}$ and $B_j$ on the choice of the Weil
$p$-number $\pi$ should be understood though 
it is omitted from the notation.  For any two square-free integers $d>1$ and $j\ge 1$, 
we write $K_{d,j}$ for the
CM field $\Q(\sqrt{d},\sqrt{-j})$. For a finite collection of
algebraic numbers $\alpha_1,\dots, \alpha_n$, the notation
$h(\alpha_1,\dots, \alpha_n)$ denotes the class number of the number field
$\Q(\alpha_1,\dots, \alpha_n)$. Particularly,
$h(\sqrt{d},\sqrt{-j})$ and $h(K_{d,j})$ have the same meaning.   \\
    
\npr {\bf Case $\pi=\pi_1\times \pi_1$.} 
For $\pi_1=\pm \sqrt{2}
\zeta_8$, one has $K=\Q(\sqrt{-1})$, $\calR_{sp}=\calR=O_K$, and
$H_{sp}(\pi)=H(\pi)=1$. For $\pi_1=\pm \sqrt{3}
\zeta_{12}$, one has $K=\Q(\sqrt{-3})$, $\calR_{sp}=\calR=O_K$, and
$H_{sp}(\pi)=H(\pi)=1$. 

For $\pi_1=\sqrt{-p}$, we have $K=\Q(\sqrt{-p})$, $\calR_{sp}=\calR$ and 
$[O_K:\calR_{sp}]=2$ or $1$ depending on $p\equiv 3 \pmod 4$ or
not. In this case we have $H_{sp}(\pi)=1,3$ for $p=2,3$, respectively
and 
\begin{equation}
  \label{eq:self_product}
  H_{sp}(\pi)=
  \begin{cases}
    h(\sqrt{-p}) & \text{for $p\equiv 1 \pmod 4$;} \\
    \left(4-\left (\frac{2}{p}\right ) \right ) h(\sqrt{-p}) 
    & \text{for $p\equiv 3 \pmod 4$ and $p>3$;}
  \end{cases}
\end{equation}
see \cite[Theorem 1.1]{yu:sp_prime}.
The contribution of the self-product cases is
\begin{equation}
  \label{eq:total_self_product}
  \sum_{\pi_1\in W^{\rm ss}_p(1)}  H_{sp}(\pi_1\times \pi_1)=
  \begin{cases}
    3, 5  & \text{for $p=2, 3$, respectively;} \\ 
    h(\sqrt{-p}) & \text{for $p\equiv 1 \pmod 4$;} \\
    \left(4-\left (\frac{2}{p}\right ) \right ) h(\sqrt{-p}) 
    & \text{for $p\equiv 3 \pmod 4$ and $p>3$.}
  \end{cases}
\end{equation}

\npr {\bf Case $\pi=\pi_1\times \pi_2$, $\pi_1\neq \pi_2$.} This
occurs only when $p=2$ or $3$. 
The following are class numbers of $B$ with $\calR_{sp}\subset
B\subset O_K$ obtained in Section~\ref{subsec:supord-Rsp-prod-case}. 

\ \\
\begin{tabular}{|c|c|c|c|c|}
\hline
$\pi=\pi_1\times \pi_2$ & $K$ &  $[O_K:\calR_{sp}]$ & $\calR_{sp}\subset
B\subset O_K$ & $h(B)$ \\ \hline
$ \sqrt{2}\zeta_4\times \pm \sqrt{2}\zeta_8$ & 
$\mathbb{Q}(\sqrt{-2}) \times \mathbb{Q}(\sqrt{-1})$ 
& 2 & $\calR_{sp}$, $O_K$  &  $1,1$ \\  \hline
$\sqrt{2}\zeta_8\times - \sqrt{2}\zeta_8$ & 
$\mathbb{Q}(\sqrt{-1}) \times \mathbb{Q}(\sqrt{-1})$ & 8 &
$\calR_{sp},B_4,B_2, O_K$ & $1,1,1,1$ \\ \hline
$ \sqrt{3}\zeta_4\times \pm \sqrt{3}\zeta_{12} $  & 
$\mathbb{Q}(\sqrt{-3}) \times \mathbb{Q}(\sqrt{-3})$ & 6 &
$\calR_{sp}, B_3, B_2, O_K$ &
$1,1,1,1$\\ \hline
$ \sqrt{3}\zeta_{12}\times -\sqrt{3}\zeta_{12}$ &
$\mathbb{Q}(\sqrt{-3}) \times \mathbb{Q}(\sqrt{-3})$ & 12 &
 $\calR_{sp}, B_4, B_3, O_K$ &
 $1,1,1,1$ \\ \hline
\end{tabular} \ \\
The orders $B_j$ are listed here for the convenience of the reader.
\begin{align*}
  B_2&=\Z[(1+\zeta_4, 0), (\zeta_4, \zeta_4)] & \quad \text{ for }
                                                \pi&=\sqrt{2}\zeta_8\times
                                                - \sqrt{2}\zeta_8;\\
  B_2&=\Z[\sqrt{-3}]\times \Z[\zeta_6] & \quad \text{ for }
  \pi&=\sqrt{3}\zeta_4\times \pm \sqrt{3}\zeta_{12};\\
  B_3&=\Z[(\sqrt{-3}, 0), (\zeta_6, \zeta_6)] &\quad \text{ for }
  \pi&=\sqrt{3}\zeta_4\times \pm\sqrt{3}\zeta_{12} \text{ or } \sqrt{3}\zeta_{12}\times -\sqrt{3}\zeta_{12};\\
  B_4&=\Z[(2,0), (\zeta_{2p}, \zeta_{2p})]  &\quad \text{ for }
  \pi&=\sqrt{p}\zeta_{4p}\times -\sqrt{p}\zeta_{4p}  \text{ and } p=2,
  3. 
\end{align*}
The contribution of other non-simple cases is 
\begin{equation}
  \label{eq:other_product}
  \sum_{\pi_1\neq \pi_2}  H_{sp}(\pi_1\times \pi_2)=
  \begin{cases}
    2\times 2+4=8  & \text{for $p=2$;} \\ 
    2\times 4+4=12 & \text{for $p=3$.} \\
  \end{cases}
\end{equation}

\npr {\bf Case $\pi\in W^{\rm ss}_p(2)$.} We have $\pi\in
\{ \pm\sqrt{2}\zeta_{24}, \pm\sqrt{5}\zeta_{5}, \sqrt{p}\zeta_8\
(p\neq 2), \sqrt{p} \zeta_3, \sqrt{p}\zeta_{12}\ (p\neq 3)\}.$  
For $\pi=\pm \sqrt{p}\zeta_n$ with $(p,n)=(5,5)$ or
$(2, 24)$, we have $\calR_{sp}=O_K$ by
Remark~\ref{rem:critical-max-order} since $n$ is critical at $p$. For
$\pi=\sqrt{p}\zeta_8$ with $p\neq 2$, we have
$K=\Q(\sqrt{p}\zeta_8)=\Q(\sqrt{-1}, \sqrt{2p})$ and
$\calR_{sp}=\Z[(\sqrt{2p}+\sqrt{-2p})/2, \sqrt{-1}]$, which is the maximal 
order in $K$ by Exercise
42(b) of \cite[Chapter 2]{MR0457396}. Therefore, 
\begin{equation}
  \label{eq:maximal_order}
  H_{sp}(\pm\sqrt{2}\zeta_{24})=H_{sp}(\pm\sqrt{5}\zeta_{5})=1,
\quad h(\sqrt{p}\zeta_8)=h(\sqrt{2p},\sqrt{-1}), \quad p\neq 2. 
\end{equation}

For $\pi=\sqrt{p}\zeta_3$, we have $K=\Q(\sqrt{p},\sqrt{-3})$ and
$\calR_{sp}=\Z[\sqrt{p}, \zeta_3]$. 
The suborders $B\subseteq O_K$ containing $\Z[\sqrt{p}]$ with the property
$[B^\times:\Z[\sqrt{p}]^\times]>1$
are classified in \cite{xue-yang-yu:num_inv}. We list the
suporders of $\calR_{sp}$ in $O_K$ and their class numbers in the
following table. 

\smallskip

\begin{tabular}{|c|c|c|c|}
\hline
$\pi=\sqrt{p}\zeta_3$  &  $[O_K:\calR_{sp}]$ & $\calR_{sp}\subset
B\subset O_K$ & $h(B)$ \\ \hline
$p=2$ &  1 & $O_K$  &  $1$ \\  \hline
$p=3$ & 3 & $\calR_{sp},\  O_K$ & $1,1$ \\ \hline
$p\equiv 3 \pmod 4,\ p\neq 3$ & 1 & $O_K$ &
$h(K)$\\ \hline
$p\equiv 1 \pmod 4$ & 4 &
 $\calR_{sp},\  O_K$ &
 $\varpi_p\, h(K), h(K)$ \\ \hline
\end{tabular} \ \\

Thus, 
\begin{equation}
  \label{eq:zeta3}
  H_{sp}(\sqrt{p}\zeta_3)=
  \begin{cases}
    1, 2  & \text{for $p=2, 3$, respectively;} \\ 
    (\varpi_p+1) h(\sqrt{p},\sqrt{-3}) & \text{for $p\equiv 1 \pmod 4$;} \\
    h(\sqrt{p},\sqrt{-3}) 
    & \text{for $p\equiv 3 \pmod 4$ and $p>3$.}
  \end{cases}
\end{equation}

For $\pi=\sqrt{p}\zeta_{12}$ ($p\neq 3$), we have
$K=\Q(\sqrt{-p},\sqrt{-3})$ and $\calR_{sp}=
\Z[\sqrt{p}\zeta_{12}, \zeta_6]=\Z[\sqrt{-p}, \zeta_6]$. We have the
following results from
Section~\ref{subsec:supord-Rsp-simple-case}. \\

\begin{tabular}{|c|c|c|c|}
\hline
$\pi=\sqrt{p}\zeta_{12}\ (p\neq 3)$  &  $[O_K:\calR_{sp}]$ &
$\calR_{sp}\subset 
B\subset O_K$ & $h(B)$ \\ \hline
$p=2$ &  1 & $O_K$  &  $1$ \\  \hline
$p\equiv 1 \pmod 4$ & 1 & $O_K$ & $h(K)$\\ \hline
$p\equiv 3 \pmod 4$ & 4 & $\calR_{sp}, O_K$ &
 $\varpi_{3p}\, h(K), h(K)$ \\ \hline
\end{tabular} \ \\

Thus, 
\begin{equation}
  \label{eq:zeta12}
  H_{sp}(\sqrt{p}\zeta_{12})=
  \begin{cases}
    1  & \text{for $p=2$;} \\ 
    h(\sqrt{-p},\sqrt{-3}) & \text{for $p\equiv 1 \pmod 4$;} \\
    (\varpi_{3p}+1) h(\sqrt{-p},\sqrt{-3})  
    & \text{for $p\equiv 3 \pmod 4\ (p\neq 3)$.}
  \end{cases}
\end{equation}

The following are the class numbers of the fields
$K=\Q(\sqrt{p} \zeta_n)$ for $n\in\{3,8,12\}$ and
$p\in\{2,3,5\}$. They are checked using Magma. \\

\begin{tabular}{|c|c|c|c|}
\hline
$h(K)$  &  $p=2$ & $p=3$ & $p=5$ \\ \hline
$\Q(\sqrt{p}\zeta_3)=\Q(\sqrt{p},\sqrt{-3})$ &  $1$ & $1$  &  $1$ \\
\hline 
$\Q(\sqrt{p}\zeta_8)=\Q(\sqrt{2p},\sqrt{-3})$ & $1$ & $2$ & $2$\\ \hline
$\Q(\sqrt{p}\zeta_{12})=\Q(\sqrt{-p},\sqrt{-3})$ & $1$ & $1$ &
 $2$ \\ \hline
\end{tabular} \ \\

We collect the contribution of simple cases. For $p=2$, we have
\begin{equation}
  \label{eq:simple2}
  H_{sp}(\sqrt{2} \zeta_3)+H_{sp}(\sqrt{2}\zeta_{12})+2
  H_{sp}(\sqrt{2}\zeta_{24})=1+1+2=4.
\end{equation}
For $p=3$, we have
\begin{equation}
  \label{eq:simple3}
  H_{sp}(\sqrt{3} \zeta_3)+H_{sp}(\sqrt{3}\zeta_{8})=1+2=3.
\end{equation}  
For $p=5$, we have 
\begin{equation}
  \label{eq:simple5}
  H_{sp}(\sqrt{5} \zeta_3)+H_{sp}(\sqrt{3}\zeta_{8})+
  H_{sp}(\sqrt{5}\zeta_{12})+2
  H_{sp}(\sqrt{5}\zeta_{5})=1+2+2+2=7.
\end{equation}
For $p\ge 7$, we have
\begin{equation}
  \label{eq:simple}
  \begin{split}
  & \sum_{\pi\neq \sqrt{p}\in W_p^{\rm ss}(2)} H_{sp}(\pi)=
  H_{sp}(\sqrt{p} \zeta_3)+H_{sp}(\sqrt{p}\zeta_{8})+
  H_{sp}(\sqrt{p}\zeta_{12}) \\
  &= 
  \begin{cases} 
  (\varpi_p+1) h(K_{p,3})+h(K_{2p,1})
  +h(K_{3p,3}), & \text{for $p\equiv 1 \pmod 4$;} \\
  h(K_{p,3})+h(K_{2p,1})+(\varpi_{3p}+1)
  h(K_{3p,3}), & \text{for $p\equiv 3 \pmod 4$.}    
  \end{cases}     
  \end{split}
\end{equation}

Let $\Delta(p)$ be the number of isomorphism classes of superspecial
abelian surfaces whose Frobenius endomorphism not equal to
$\pm\sqrt{p}$. Then we have
\begin{equation}
  \label{eq:Delta}
  \Delta(p)=\sum_{\pi\in W^{\rm ss}_p(2), \pi\neq \sqrt{p}}
  H_{sp}(\pi)+\sum_{\pi_1\times \pi_2, \pi_1\neq \pi_2}
  H_{sp}(\pi_1\times \pi_2)+\sum_{\pi_1\in W^{\rm ss}_p(1)}
  H_{sp}(\pi_1\times \pi_1). 
\end{equation}

Collecting the results (\ref{eq:total_self_product}),
 (\ref{eq:other_product}), (\ref{eq:simple2})
(\ref{eq:simple3}), (\ref{eq:simple5}) and (\ref{eq:simple}), 
we obtain the following result.
 
\begin{thm}\label{sp.2} \ 
  \begin{enumerate}
  \item The number $\Delta(p)$ is $15,20,9$ for $p=2,3,5$,
  respectively.
  \item For $p>5$ and $p\equiv 1 \pmod 4$, we have    
  \begin{equation}
    \label{eq:Delta_1mod4}
    \Delta(p)=(\varpi_p+1) h(K_{p,3})+h(K_{2p,1})
  +h(K_{3p,3})+h(\sqrt{-p}),
  \end{equation}
  where $\varpi_{p}$ is defined in (\ref{eq:varpi_d}).
  \item For $p>5$ and $p\equiv 3 \pmod 4$, we have 
  \begin{equation}
    \label{eq:Delta_3mod4}
   \Delta(p)=h(K_{p,3})+h(K_{2p,1})+(\varpi_{3p}+1)
   h(K_{3p,3})+\left(4-\left (\frac{2}{p}\right ) \right ) h(\sqrt{-p}),
  \end{equation}
  where $\varpi_{3p}$ is defined in (\ref{eq:varpi_d}).
  \end{enumerate}
\end{thm}

Theorem~\ref{1.3} then follows from Theorems~\ref{1.2} and \ref{sp.2}.

\begin{remark}\label{sp.3}
  Based on our computation we observe that 
  the endomorphism ring of a
  superspecial abelian surface over $\Fp$ may be a non-maximal order, 
  or even non-maximal at $p$. For example, when $p=3$ and $\pi=\sqrt{3} \zeta_3$,
  the order $\calR_{sp}$, which occurs as the
  endomorphism ring of a superspecial abelian surface \cite[Theorem
  6.1]{waterhouse:thesis},  has index $3$ in the maximal order.
\end{remark}

\subsection{Asymptotic behavior of $|\Sp_2(\Fp)|$}
\label{sec:sp.3}


\def\Mass{\mathrm {Mass}}
We now determine the asymptotic behavior of the size of $\Sp_2(\Fp)$
as the prime $p$ goes to infinity.  For simplicity, let $F=\Q(\sqrt{p})$. 
By Theorem~\ref{1.3}, $|\Sp_2(\Fp)|$ is expressed as a 
linear combination of
$\zeta_F(-1) h(F)$, $h(\sqrt{-p})$, and class numbers of certain biquadratic CM
fields.  The term $c \zeta_F(-1) h(\sqrt{p})$ (for a suitable
constant $c$) comes from the contribution of 
the isogeny class corresponding to the Weil $p$-number
$\pi=\sqrt{p}$. More precisely, it arises from the mass part in the
Eichler class number formula for the calculation of $H(\sqrt{p})$. 
We recall from Theorem~\ref{1.2} that the mass part 
for $p>5$ is 
\begin{equation}
  \label{eq:mass}
  \Mass(p)=
  \begin{cases}
   \frac{1}{2} \zeta_F(-1) h(F) & \text{for $p\equiv 3 \pmod 4$;} \\
   8 \zeta_F(-1) h(F)  & \text{for $p\equiv 1 \pmod 8$;} \\
   \frac{1}{2}(15\varpi_p+1)\zeta_F(-1) h(F) & \text{for $p\equiv 5 \pmod
   8$}. \\  
  \end{cases}
\end{equation}
In \cite[Theorem 6.3.1]{xue-yang-yu:ECNF} we showed that the mass part
$\Mass(p)$ is the main term of $H(\sqrt{p})$. 
It is expected that $\Mass(p)$ is
also the main term of $|\Sp_2(\Fp)|$. This is true and we have the
following asymptotic formula for the size $\Sp_2(\Fp)$. 
      
\begin{prop}\label{sp.4}
  We have 
\[ \lim_{p\to \infty} \frac{|\Sp_2(\Fp)| 
 }{\Mass(p)}=1. \]  
\end{prop}
\begin{proof}
It is enough to show that $\lim_{p\to
  \infty}h(\sqrt{-p})/h(F)\zeta_F(-1)=0$, and for all the biquadratic CM-fields
$K_{d,j}$ appearing in the formula of $|\Sp_2(\Fp)|$,  \[\lim_{p\to \infty}
h(K_{d,j})/h(F)\zeta_F(-1)=0.\] 
The above limit has been verified for the pairs $(d, j)$ with $d=p$ and
$j=1,2,3$ in \cite[Theorem~6.3.1]{xue-yang-yu:ECNF}, and it remains to
consider the pairs $(2p,1)$ and $(3p, 3)$. 

Recall that the
discriminant of $F$ is denoted by $\grd_F$, which is either $p$ or
$4p$. Using the function equation and the trivial inequality
$\zeta_F(2)>1$, we have
$\zeta_F(-1)>c_1 (\grd_F)^{3/2}$ for a constant $c_1>0$.  
For any CM-field $K$, let $h^{-}(K)$ be the relative class
number of $K$  defined as $h(K)/h(K^+)$, where $K^+$ is the maximal
totally real subfield of $K$. 
By \cite[Lemma~4]{Horie-Horie-1990}, when $K$ range over a sequence of
CM-fields with bounded degree and $\grd_K\to \infty$, we have 
\begin{equation}
  \label{eq:Horie-Horie}
\lim_{\grd_K\to \infty} (\log
  h^-(K))/(\log \sqrt{\grd_K/\grd_{K^+}})=1.  
\end{equation}
In particular, applying this to the quadratic imaginary fields
$\Q(\sqrt{-p})$, we obtain that $h(\sqrt{-p})/\zeta_F(-1)\to 0$ as $p\to \infty$. 

Assume $(d,j)=(2p, 1)$ or $(3p,3)$. One calculates that
$\grd_{K_{d,j}}/\grd_{K_{d,j}^+}\leq 32p$. Let $\epsilon_d$ be the
fundamental unit of the quadratic real field $\Q(\sqrt{d})$.  By Siegel's theorem \cite[Theorem
  15.4, Chapter 12]{MR665428}, the growth of
$h(K_{d,j}^+)=h(\sqrt{d})$ satisfies the following formula \[\lim_{d\to \infty} (\log
  h(\sqrt{d})\log\epsilon_d)/(\log \sqrt{d})=1.\]
Note that $\epsilon_d$ is bounded below by $(1+\sqrt{5})/2$ for all
$d$. Recall that $h(K_{d,j})=h^-(K_{d,j})h(\sqrt{d})$.  Combining these bounds yields that $h(K_{d,j})/\zeta_F(-1)\to 0$ as $p$ goes to infinity.
\end{proof}



\section{Galois cohomology of an arithmetic group}
\label{sec:gal_coh}

\subsection{Galois cohomology and conjugacy classes}
\label{sec:gal_coh.1} 
We refer to \cite[Section I.5]{Serre-Galois-coh} for the definition of
nonabelian Galois cohomology. Let $\Gamma_{\F_q}=\Gal(\Fqbar/\F_q)$ be
the absolute Galois group of $\F_q$, and $G$ a group with discrete
topology on which $\Gamma_{\F_q}$ acts continuously. Let $\sigma_q$ be the
arithmetic Frobenius automorphism of $\Fqbar$, which raises each element
of $\overline{\F}_q$ to its $q$-th power. The  group
$\Gamma_{\F_q}$ is isomorphic to the profinite group
$\widehat{\Z}=\varprojlim_{n\in \bbN}\zmod{n}$  with canonical
generator $\sigma_q$.  Each $1$-cocycle $(a_\sigma)_{\sigma\in \Gamma_{\F_q}}$ is uniquely
determined by its value $x=a_{\sigma_q}\in G$ at $\sigma_q$.  An element
of $G$ is called a \emph{$1$-cocycle element} if it arises from
a $1$-cocycle in this way. We will identify the set of $1$-cocycles
$Z^1(\Gamma_{\F_q}, G)$ with the subset of $1$-cocycle elements of $G$. Two $1$-cocycle elements $x, y\in Z^1(\Gamma_{\F_q},
G)$ define the same cohomology class if and only if they are 
$\sigma_q$-conjugate (denoted by $x\sim_{\sigma_q} y$), i.e., there exists
$z\in G$ such that $x=z^{-1}y\sigma_q(z)$. Write $[x]_{\sigma_q}$
for the $\sigma_q$-conjugacy
class of $x\in G$, and $B(G)$ for the set of all $\sigma_q$-conjugacy
classes of $G$.  Then 
\[ H^1(\Gamma_{\F_q}, G)= \{[x]_{\sigma_q}\in B(G)\mid
 x\in Z^1(\Gamma_{\F_q},G) \}\subseteq B(G).\]
If the action of $\Gamma_{\F_q}$ on $G$ is
trivial, then $B(G)$ is reduced to the set $\Cl(G)$ of (the usual)
conjugacy classes of $G$.  Define $\Cl_0(G):=\{[x]\in \Cl(G)\mid x \text{ is
  of finite order}\}\subseteq \Cl(G)$.


\begin{lemma}\label{lm:gal.1}
Assume that the action of $\Gamma_{\F_q}$ on $G$ factors through a finite
quotient $\Gal(\F_{q^N}/\F_q)$. We have 
\begin{gather*}
Z^1(\Gamma_{\F_q},G)=\{x\in G\mid x\sigma_q(x)\cdots\sigma_q^{N-1}(x) \text{ is of finite order}\,\}.
\end{gather*}
In particular, if the action of $\Gamma_{\Fq}$ on $G$ is trivial, then $H^1(\Gamma_{\F_q}, G)=\Cl_0(G)$.
\end{lemma}
\begin{proof}
  This follows directly from Exercise~2 of
  \cite[Section~I.5.1]{Serre-Galois-coh}. 
\end{proof}

\subsection{Abelian varieties over finite fields and twisted forms}\label{subsec:galois-descent}
Let $X_0$ be an abelian variety over $\F_q$  with Frobenius
endomorphism $\pi_{X_0}\in \End_{\F_q}(X_0)$. Set $\ol{X}_0=X_0\otimes
\Fqbar$, and $G=\Aut(\ol{X}_0)$. The Galois group
$\Gamma_{\F_q}$ acts on $\End(\ol{X}_0)$ as follows  (see \cite[Proposition
4.3]{yu:superspecial}) 
\begin{equation}
  \label{eq:gal.4}
  \sigma_q(x)=\pi_{X_0} x \pi_{X_0}^{-1}, \quad \forall\, x\in \End(\ol{X}_0),
\end{equation}
where the conjugation by $\pi_{X_0}$ is taken inside
$\End^0(\ol{X}_0)$. As $\End(\ol{X}_0)$ is a free $\Z$-module of
finite rank, the action of $\Gamma_{\F_q}$ factors
through a finite quotient $\Gal(\F_{q^N}/\F_q)$, and hence $(\pi_{X_0})^N$
is central in $\End(\ol{X}_0)$.  


Recall that an \emph{$(\Fqbar/\F_q)$-form} of
$X_0$ is an abelian varieties $X$ over $\F_q$ such that $\ol{X}:=X\otimes \Fqbar$ is $\Fqbar$-isomorphic to $\ol{X}_0$. 
Let $E(\Fqbar/\F_q, X_0)$ be the set of $\F_q$-isomorphism classes of $(\Fqbar/\F_q)$-forms of $X_0$. By
\cite[Section~III.1.3]{Serre-Galois-coh}, there is a canonical
bijection of pointed sets 
\begin{equation}
  \label{eq:5}
  \theta: E(\Fqbar/\F_q, X_0)\isoto H^1(\Gamma_{\Fq},G), 
\end{equation}
sending the $\F_q$-isomorphism class of $X_0$ to
the trivial class. 
The map $\theta$ is induced from mapping each $\Fqbar$-isomorphism $f: \ol{X}_0\to \ol{X}$
to the 1-cocycle element $x=f^{-1}\sigma_q(f)\in G$.   The injectivity of
$\theta$ follows purely from cohomological formalism, and the
surjectivity is a consequence of Weil's Galois descent. 

An isomorphism $f$ of abelian varieties as above induces an
isomorphism 
\begin{equation}
  \label{eq:isom-end-alg}
 \alpha_f:\End(\ol X)\simeq \End(\ol X_0),
\quad y\mapsto f^{-1} y f.   
\end{equation}
The Frobenius
endomorphisms $\pi_{X_0}$ and $\pi_X$
are related by the following
commutative diagram (see \cite[(4.2)]{yu:superspecial}):
\begin{equation}
  \label{eq:gal_coh.10}
  \begin{CD}
    \ol X_0 @>{f}>> \ol X \\
    @V{\pi_{X_0}}VV @VV{\pi_X}V \\
    \ol X_0 @>{\sigma_q(f)}>> \ol X.
  \end{CD}
\end{equation}
We compute
\begin{equation}
  \label{eq:6}
 \alpha_f(\pi_X)=f^{-1} \pi_X f=f^{-1} \sigma_q(f) \pi_{X_0} =x \pi_{X_0}.   
\end{equation}
   
Note that for $x, y, z\in G$, 
\[x=z^{-1}y\sigma_q(z)\  \Leftrightarrow\ 
x\pi_{X_0}=z^{-1}(y\pi_{X_0})z. \]
Hence there is a well-defined injective map 
\begin{equation}
  \label{eq:7}
\Pi:  B(G)\hookrightarrow \End(\ol{X}_0)/G, \quad
[x]_{\sigma_q}\mapsto [x\pi_{X_0}], 
\end{equation}
where $\End(\ol{X}_0)/G$ denotes orbits of
$\End(\ol{X}_0)$ under the right action of $G$ by conjugation. In a sense, the image of
$H^1(\Gamma_{\F_q}, G)$ under $\Pi$ consists of the conjugacy classes
of Frobenius endomorphisms of members of $E(\Fqbar/\F_q, X_0)$.

We can also work in the category of abelian varieties up to
isogeny and study the $(\Fqbar/\F_q)$-forms of
the isogeny class $[X_0]$. Thus we pass from isomorphisms of abelian
varieties to quasi-isogenies, and endomorphism rings  to endomorphism
algebras, etc.  Let $E^0(\Fqbar/\F_q,
[X_0])$ be the set of $\F_q$-isogeny classes of abelian varieties $[X]$
 such that $\ol{X}$ is isogenous to $\ol{X}_0$ over $\Fqbar$, and
 $G_\Q=\End^0(\ol{X}_0)^\times$.  Many previous constructions can
be carried over. In particular, both (\ref{eq:gal_coh.10}) and 
(\ref{eq:6}) hold true for any quasi-isogeny $f:
\ol{X}_0\to \ol{X}$, and one obtains a 1-cocycle element
$x=f^{-1}\sigma_q(f)\in G_\Q$ as before. This gives a canonical injective map 
\begin{equation}
  \label{eq:8}
\theta: E^0(\Fqbar/\F_q, [X_0])\hookrightarrow H^1(\Gamma_{\Fq},G_\Q),     
\end{equation}
which fits into a commutative diagram 
\begin{equation}
  \label{eq:9}
  \begin{CD}
    E(\Fqbar/\F_q, X_0) @>{\cong}>{\theta}> H^1(\Gamma_{\Fq},G) \\
    @VVV @VVV \\
    E^0(\Fqbar/\F_q, [X_0]) @>{\theta}>> H^1(\Gamma_{\Fq},G_\Q).
  \end{CD}
  \end{equation}
The left vertical map sends the $\F_q$-isomorphism class of $X$ to its
$\F_q$-isogeny class $[X]$, and the right vertical map is induced from the
inclusion of $\Gamma_{\F_q}$-groups $G\subset G_\Q$. Thus
 (\ref{eq:9}) endows a geometric
meaning for this cohomological map. 


We complete the picture by showing that the map $\theta$ in
(\ref{eq:8}) is surjective and thus a bijection of pointed sets as
stated in Theorem~\ref{thm:intro-descent-isog}. Recall that the action of $\Gamma_{\F_q}$ on
$\End^0(\ol X)$ factors through $\Gal(\F_{q^N}/\F_q)$ for a fixed
$N\in \bbN$. Without lose of generality, assume that $X_0$ is
$\F_{q^N}$-isotypical, i.e., $X_0\otimes\F_{q^N}$ is isogenous to $(Y_N)^d$, where $Y_N$
is an absolutely simple abelian variety over $\F_{q^N}$ with
$\End(Y_N)=\End(\ol Y_N)$. Equivalently, we assume
that the multiple Weil $q$-number
$\pi_{0,1}^{t_1}\times \dots \times \pi_{0,u}^{t_u}\in MW_q$
corresponding to the $\F_q$-isogeny class $[X_0]$ satisfies that
$\pi_{0,1}^N=\pi_{0,2}^N=\cdots= \pi_{0,u}^N$ after suitable
conjugation, and  $\Q((\pi_{X_0})^N)\subset \End^0(X_0)$ is a
field which coincides with $\Q((\pi_{X_0})^{sN})$
for all $s\in \bbN$.
Then $\End^0(\ol{X}_0)=\Mat_d(\End^0(\ol Y_N))$, and
$\End^0(\ol Y_N)$ is a central division algebra over $\Q((\pi_{X_0})^N)$.   For simplicity, let $D=\End^0(\ol Y_N)$ and
$K_0=\Q((\pi_{X_0})^N)$. Then $G_\Q=\End^0(\ol{X}_0)^\times=\GL_d(D)$.



\begin{lem}\label{lem:coh-to-conj-class}
  There is a bijection between $H^1(\Gamma_{\F_q}, G_\Q)$ and the following
  subset of conjugacy classes of $\Cl(G_\Q)$:
  \begin{equation}
    \label{eq:1-cocycle-cond}
 \mathscr{C}(\pi_{X_0})=\{[\ul\pi]\in \Cl(G_\Q)\mid \exists M\in \bbN: \
 \ul\pi^{NM}=\pi_{X_0}^{NM}\}\subset \Cl(G_\Q). 
  \end{equation}
\end{lem}
\begin{proof}
 Since $\pi_{X_0}\in G_\Q$, the map $\Pi$ in (\ref{eq:7})
defines a bijection
\[\Pi: B(G_\Q)\isoto \Cl(G_\Q),\quad [x]_{\sigma_q}\mapsto [x\pi_{X_0}].  \]
Let $\pi_x=x\pi_{X_0}$ for each $x\in G_\Q$. Then
 \[x\sigma_q(x)\cdots
\sigma_q^{N-1}(x)=(x\pi_{X_0})^N(\pi_{X_0})^{-N}=(\pi_x)^N(\pi_{X_0})^{-N}. \]
By Lemma~\ref{lm:gal.1},  $x\in Z^1(\Gamma_{\F_q},G_\Q)$ if and only
if $(\pi_x)^N=(\pi_{X_0})^N\xi$ for some  $\xi\in G_\Q$ of
finite order, or equivalently, $\pi_x^{NM}=(\pi_{X_0})^{NM}$ for some $M\in \bbN$.
\end{proof}




Any $\ul{\pi}\in G_\Q$ with $[\ul{\pi}]\in \mathscr{C}(\pi_{X_0})$
is semisimple, as $\ul{\pi}^{NM}=(\pi_{X_0})^{NM}$ lies in the center of the simple
$\Q$-algebra $\End^0(\ol{X}_0)$.    The minimal polynomial of
$\ul{\pi}$ factorizes as a product of distinct irreducible
polynomials over $\Q$:
\begin{equation}
  \label{eq:factor-min-poly}
 P(t)=\prod_{i=1}^r P_i(t)\in \Q[t].
\end{equation}
For all $\ul \pi'$ in the conjugacy class $[\ul \pi]$, the $\Q$-subalgebra
$K_{\ul\pi'}:=\Q(\ul\pi')\subset \End^0(\ol{X}_0)$ is canonically
isomorphic to $K:=\Q[t]/(P(t))$ via the map $\ul \pi'\mapsto
t$.  Since
$\pi_{X_0}^{NM}=\ul\pi^{NM}$, the field
$K_0=\Q(\pi_{X_0}^{NM})$ can be identified with  the $\Q$-subalgebra of $K$
generated by $t^{NM}$, thus providing a
$K_0$-algebra structure on $K$. By (\ref{eq:factor-min-poly}), $K$
factorizes as a products of number fields
\begin{equation}
  \label{eq:factor-field} 
K=K_1\times \cdots \times K_r, \quad
\text{with}\quad  K_i=\Q[t]/(P_i(t))\supseteq K_0.
\end{equation}
By an abuse of notation, we regard $\pi_{X_0}^N$ as a Weil
$q^N$-number via a embedding $K_0\hookrightarrow \Qbar$. 
Then for each $1\leq i\leq r$, the roots of $P_i(t)$ in $\Qbar$ is a
conjugacy class of Weil $q$-numbers such that one of its
representative $\pi_i$ satisfies $\pi_i^{NM}=\pi_{X_0}^{NM}$.  Therefore,  given
$\ul\pi\in \mathscr{C}(\pi_{X_0})$, we find $r$
Weil $q$-numbers representing distinct conjugacy classes
\begin{equation}
  \label{eq:Weil-num-from-cohomology}
  \{\pi_1, \cdots, \pi_r\}\quad \text{with}  \quad
  \pi_i^{NM}=\pi_{X_0}^{NM}\quad \text{for some } M\in \bbN \text{ and
    all } 1\leq i \leq r. 
\end{equation}

Next, we fix $P(t)\in \Q[t]$ as above, and produce a discrete invariant for every conjugacy class
$[\ul\pi]\in \mathscr{C}(\pi_{X_0})$ with minimal polynomial
$P(t)$.  Let $V=D^d$ be the right
vector space over $D$ of column vectors.  Then $\End_D(V)=\Mat_d(D)$
acts on $V$ from the left by the usual matrix multiplication. We have
a canonical $K_0$-algebra embedding $K\embed \End_D(V)$ sending $K$ to
$K_\pi$.  Thus $\ul\pi$ endows a
faithful $(K,D)$-bimodule structure on $V$, denoted by
$V_{\ul\pi}$. 
By (\ref{eq:factor-field}), there is a decomposition of $V$ into right $D$-subspaces:
\begin{equation}
  \label{eq:bimodule-decomposition}
 V=\bigoplus_{i=1}^r V_i, \quad d_i=\dim_D V_i.   
\end{equation}
The action of $K_i$ on $V_i$ gives rise to a $K_0$-embedding
$K_i\embed \End_D(V_i)=\Mat_{d_i}(D)$.
We study each of the embeddings individually first.

\begin{lem}\label{lem:minimal-embedding}
Let $\pi\in W_q$ be a Weil $q$-number such
that $\pi^{NM}=\pi_{X_0}^{NM}$ for some integer $M\in \bbN$, and
$X_\pi$ a simple abelian variety over $\F_q$ in the isogeny class corresponding to
$\pi$. Let $e=e(\pi)$ be the smallest integer such that there is an
$K_0$-embedding $\Q(\pi)\embed \Mat_e(D)$. Then $\ol
X_\pi=X_\pi\otimes \Fqbar$ is
isogenous to  $(\ol Y_N)^e$,  and $\End^0(X_\pi)$ is isomorphic to the
centralizer $C_\pi$ of $\Q(\pi)$ in $\Mat_e(D)$. 
  \end{lem}
  \begin{proof}
    Since $\pi^{NM}=\pi_{X_0}^{NM}$, there exists an isogeny $\ol X_\pi\to
    (\ol Y_N)^e$ for some $e\in \bbN$, which gives an identification of $\End^0(\ol X_\pi)$ with $\Mat_e(D)=\End^0((\ol Y_N)^e)$
    in the same way as (\ref{eq:isom-end-alg}).  Thus we obtain a $K_0$-embedding $\Q(\pi)\embed \Mat_e(D)$, and $\End^0(X_\pi)$ is
    recovered as the $\Gamma_{\F_q}$-invariants of $\End^0(\ol
    X_\pi)$, or equivalently, the centralizer $C_\pi$ of $\Q(\pi)$ in
    $\Mat_e(D)$ by (\ref{eq:gal.4}).  On the other hand, $C_\pi$ is also the endomorphism
    algebra of the $(\Q(\pi), D)$-bimodule $D^e$. Now the minimality
    of $e$ follows from the fact that $C_\pi=\End^0(X_\pi)$ is a
    division algebra. 
  \end{proof}
Given $e'\in \bbN$,  a $K_0$-embedding of $\Q(\pi)$ into the simple
  algebra $\Mat_{e'}(D)$ exist if and only if $e(\pi)$ divides $e'$.
  Therefore, every $d_i$ in
  (\ref{eq:bimodule-decomposition}) is of the form $m_i e(\pi_i)$
  for some positive integer $m_i\in \bbN$ subjecting to the condition
\begin{equation}
  \label{eq:admissible-cond}
  m_1 e(\pi_1)+\dots+m_r e(\pi_r)=d.
\end{equation}
We shall call the $r$-tuple 
$\ul m=(m_1,\dots, m_r)$ the \emph{type} of the $(K, D)$-bimodule
$V_{\ul \pi}$ or simply the \emph{type} of $\ul \pi$. 

\begin{lem}\label{lem:bijections-str}
  There are natural bijections between the following sets:
  \begin{enumerate}
  \item  the set of conjugacy classes
$[\ul\pi]\in \mathscr{C}(\pi_{X_0})$ with minimal polynomial
$P(t)$;
\item  the set of $G_\Q$-conjugacy classes of $K_0$-embedding
  $K\embed \End_D(V)$;
\item the set of isomorphism classes of faithful $(K,D)$-bimodule structures on $V$;
\item the set of $r$-tuples 
$\ul m=(m_1,\dots, m_r)\in \bbN^r$ satisfying (\ref{eq:admissible-cond}). 
  \end{enumerate}
\end{lem}
\begin{proof}
The bijection between between (1) and (2) is established by the map sending each $K_0$-embedding
  $\phi: K=\Q[t]/(P(t))\embed \End_D(V)$ to $\pi=\phi(t)$.   Every faithful $(K,D)$-bimodule structure on $V$ is given by a
  $K_0$-embedding $\phi: K\embed \End_D(V)$. Two
  such embeddings define isomorphic structures if and only if they are
  conjugate by an element of $G_\Q$.  Hence (2) is bijective
  to (3). 

  The proof that (2) is bijective to (4) is similar to that of \cite[Proposition~3.2]{shih-yang-yu}
  and is omitted.
\end{proof}



 \begin{thm}\label{thm:galois-descent-isog-classes}
Each cohomology class $[x]_{\sigma_q}\in H^1(\Gamma_{\F_q},G_\Q)$
determines a unique conjugacy class of multiple
  Weil $q$-number $\pi_1^{m_1}\times \cdots \times \pi_r^{m_r}\in
  MW_q$ such 
  that 
\begin{itemize}
\item $\pi_i^{NM}=\pi_{X_0}^{NM}$ for some
  $M\in \bbN$ and all $1\leq i \leq r$; 
\item $\ul m=(m_1,\dots, m_r)\in \bbN^r$ satisfies
 (\ref{eq:admissible-cond}). 
  \end{itemize} 
In particular, the map $\theta$ in (\ref{eq:8}) is  a bijection  of pointed sets. 
 \end{thm}
\begin{proof}
Given $[x]_{\sigma_q}\in H^1(\Gamma_{\F_q},G_\Q)$, we produce the
desired multiple Weil $q$-number by combing the type $\ul m=(m_1,
\cdots, m_r)$ of
$[\pi_x]\in \mathscr{C}(\pi_{X_0})$ and the set $\{\pi_1, \cdots, \pi_r\}$ determined by $[\pi_x]$
as in (\ref{eq:Weil-num-from-cohomology}).  Let $X=\prod_{i=1}^r (X_{\pi_i})^{m_i}$ be an abelian variety over
$\F_q$ corresponding to $\pi_1^{m_1}\times \cdots \times \pi_r^{m_r}$. Then $\ol X$ is
isogenous to $\ol{X}_0$ by Lemma~\ref{lem:minimal-embedding} and
(\ref{eq:admissible-cond}). Identify $\End^0(\ol X)$ with $\End^0(\ol X_0)$
via an isogeny $f: \ol{X}_0\to \ol{X}$ as in
(\ref{eq:isom-end-alg}). The conjugacy class of $\alpha_f(\pi_X)\in
G_\Q$ is independent of the choice of $f$. By the construction,
$\alpha_f(\pi_X)$ is a semisimple element with the same minimal
polynomial and type as $\pi_x=x\pi_{X_0}$.  It follows from
Lemma~\ref{lem:bijections-str} that they must lie in the same conjugacy class of $G_\Q$. We conclude that 
$\theta$ is surjective by Lemma~\ref{lem:coh-to-conj-class}. 
 \end{proof}

\subsection{Superspecial abelian varieties and the parity property}
\label{sec:parity}

We apply the previous construction to the study of superspecial
abelian varieties over finite fields. Let $E_0$ be a supersingular elliptic curve
over the prime finite field $\Fp$ 
whose Frobenius endomorphism $\pi_0$ satisfying
$\pi_0^2+p=0$ (Recall that $\sqrt{-p}\in W_p^{\rm ss}(1)$ for all $p$ by
Proposition~\ref{prop:dim.5}). Let $\calO:=\End(E_0\otimes \Fpbar)$ be
the endomorphism 
ring of $E_0\otimes \Fpbar$; this is a maximal order in the unique
quaternion $\Q$-algebra $D=B_{p,\infty}$ ramified exactly 
at $\{p,\infty\}$. 
Take $X_0=E_0^d$ and $\ol X_0:=X_0 \otimes \Fpbar$ for $d\ge 1$. 
Then $\End(\ol X_0)=\Mat_d(\calO)$. 
In what follows we denote by
\[ G:=\Aut(\ol X_0)=\GL_d(\calO) \] 
the automorphism group of $\ol X_0$.   Consider $\calO$ as a subring of $\Mat_d(\calO)$ by the diagonal
  embedding and view $\pi_0$ as an element in $\Mat_d(\calO)$. Then
  the action of $\Gamma_{\Fp}$ on $G=\GL_d(\calO)$ is given by
\begin{equation}
 \label{eq:gal-act-ss}
  \sigma_p(x)=\pi_0 x \pi_0^{-1}, \quad x\in G.
\end{equation}
We will also write $G_\Q$ for the group $\GL_d(D)$. 




Recall that $\Sp_d(\Fq)$ denotes the set of isomorphism classes of
 $d$-dimensional superspecial abelian varieties over $\Fq$. For the
classification of superspecial abelian varieties over the algebraic
closure $\overline{\F}_q$ of $\Fq$, 
we have the following result, due to Deligne, Shioda and Ogus
(cf.~\cite[Section 1.6, p.~13]{li-oort}). 

\begin{thm}\label{gal.sp}
For any integer $d\ge 2$, 
there is only one isomorphism class of $d$-dimensional superspecial
abelian varieties over any \ac  field of \ch $p>0$. 
\end{thm}

According to this theorem, any $d$-dimensional 
superspecial abelian variety over $\Fq$
is an $(\overline{\F}_q/\Fq)$-form of $X_0\otimes \Fq$.
Thus we obtain a natural bijection by (\ref{eq:5})
\begin{equation}
  \label{eq:gal.3}
   H^1(\Gamma_{\Fq},G)\simeq \Sp_d(\Fq), \quad d>1, 
\end{equation}
which sends the trivial class to the isomorphism 
class of $X_0\otimes \Fq$. The set $\Sp_d(\Fq)$ is partitioned
into isogeny classes:
\begin{equation}
  \label{eq:partition-ss-class}
  \Sp_d(\Fq)=\coprod_{\pi\in MW^{\rm ss}_q(d)} \Sp(\pi). 
\end{equation}

\begin{thm}\label{gal.1}
  Let $q=p^a$ and $q'=p^{a'}$ be powers of the prime number $p$ such
  that $a\equiv a' \pmod {2}$. For any integer $d\ge 1$, there are 
  natural bijections 
  \begin{gather}
   H^1(\Gamma_{\Fq},G) \simeq H^1(\Gamma_{\F_{q'}},G),
   \label{eq:gal.parity}\\ 
   \Sp_d(\Fq) \simeq \Sp_d(\F_{q'}). \label{eq:sp.parity} 
  \end{gather}
\end{thm}
\begin{proof}
  If $d=1$, then (\ref{eq:sp.parity}) has been
  proven in Section~\ref{sec:curve.2}; see Theorem~\ref{curve.3} and
  Remark~\ref{curve.4}. If $d>1$, then (\ref{eq:sp.parity}) follows
  from (\ref{eq:gal.3}) and (\ref{eq:gal.parity}). Therefore, it
  remains to 
  prove (\ref{eq:gal.parity}). 

   Since $\pi_0^2$ is a central element, the element $\sigma_p^2$ acts
   trivially on $G$ by (\ref{eq:gal-act-ss}). Thus $\sigma_q(x)= \sigma_{q'}(x)$ for all $x\in G$.
  This together
  with the
  canonical isomorphism $\Gamma_{\Fq}\simeq \Gamma_{\F_{q'}}$ (sending
  $\sigma_{q}\mapsto 
  \sigma_{q'}$) yields
  a natural bijection 
  $H^1(\Gamma_{\Fq}, G)\simeq H^1(\Gamma_{\F_{q'}},
  G)$. 
The theorem is proved.
\end{proof}

\begin{rem}\label{rem:parity-isogeny-class}
  By the same token, we have a natural bijection 
\begin{equation}
   H^1(\Gamma_{\Fq},G_\Q) \simeq H^1(\Gamma_{\F_{q'}},G_\Q). 
\end{equation}
Thus by Theorem~\ref{thm:galois-descent-isog-classes},  there is also a \emph{natural} bijection between the isogeny
classes of supersingular abelian varieties over $\F_q$ and those over $\F_{q'}$. 
This can be made explicit in terms of multiple Weil numbers. 
The Frobenius endomorphism of $X_0\otimes \Fq$ is $\pi_0^a$.
Hence the Frobenius endomorphisms of the isogeny class
corresponding to a cohomology class $[x]_{\sigma_q}\in H^1(\Gamma_{\Fq},G_\Q)$ is given
by the conjugacy class $[x\pi_0^a]$ by
(\ref{eq:6}).  Without lose of generality, assume that $a-a'=2s\geq 0$.  If $\pi=\pi_1^{m_1}\times
  \cdots \times
  \pi_r^{m_r}$ is a multiple Weil $q$-number determined by
  $[x]_{\sigma_q}$, then the corresponding multiple Weil $q'$-number 
is $\wt\pi=\wt{\pi}_1^{m_1}\times
  \cdots \times
  \wt{\pi}_r^{m_r}$, with $\wt{\pi}_i=(-p)^{-s}\pi_i$ for all $1\leq i
  \leq r$. 

By the commutative diagram (\ref{eq:9}),  the bijection
(\ref{eq:sp.parity}) preserves isogeny classes in the sense that there
is a natural bijection
\begin{equation}
  \label{eq:bijection-parity}
  \Sp(\pi)\simeq \Sp(\wt\pi)\qquad \forall\,  \pi\in MW^{\rm ss}_q(d).
\end{equation}
\end{rem}


\begin{thm}\label{gal.odd}
  Let $q=p^{2s+1}$ be an odd power of $p$. Let $Y_0$ be a fixed
  supersingular abelian variety over $\Fq$ and $\pi=\pi_1^{m_1}\times
  \cdots \times
  \pi_r^{m_r}$ the corresponding multiple Weil $q$-number. Let $V$ and
  $K$ be as in Theorem~\ref{prop:sp.1}, 
  and set $\calR_{sp}:=\Z[\wt \pi_0, p \wt
  \pi_0^{-1}]\subset K$, where $\wt
  \pi_0=(-p)^{-s}(\pi_1,\dots,\pi_r)$. Assume that $K$ has no real
  places. Then there is a natural bijection between the set $\Sp(\pi)$
  of isomorphism classes of superspecial abelian varieties in the
  isogeny class $[Y_0]$ and the set of 
  isomorphism classes of $\calR_{sp}$-lattices in $V$.   
\end{thm}
\begin{proof}
   This follows from Theorems \ref{prop:sp.1} and \ref{gal.1}. 
\end{proof}



The above theorem provides an approach for computing the size of
$\Sp_d(\Fq)$ explicitly in the odd exponent case, subject to 
the condition that $K$ has no totally real factors.  
For the rest of this section we shall describe 
$H^1(\Gamma_{\Fq},
\GL_d(\calO))$ (and hence $\Sp_d(\Fq)$) 
when $q$ is an even power of $p$.

\subsection{A description of $H^1(\Gamma_{\Fq}, \GL_d(\calO))$ with
  even exponent.}
\label{sec:gal_coh.2}
\def\Emb{{\rm Emb}}

In what follows we assume that $q=p^a$ is an even power of $p$
and $X_0=E_0^d\otimes \F_q$ with $d\ge 2$. 
The Frobenius endomorphism $\pi_{X_0}=(-p)^{a/2}$ lies
in the center of $\End(\ol X_0)=\Mat_d(\calO)$. 
Hence $\Gamma_{\Fq}$ acts trivially on 
the group $G:=\GL_d(\calO)$ by (\ref{eq:gal-act-ss}). 
Then it follows from Lemma~\ref{lm:gal.1} that $H^1(\Gamma_{\Fq}, G)$ can be identified with 
the set  
$\Cl_0(G)$ of conjugacy classes of elements in
$G$ of finite order. 
We shall give a lattice
description for $\Cl_0(G)$ and hence for $\Sp_d(\Fq)$ by the previous
correspondence. See Theorem~\ref{gal_coh.3} for details.

Suppose $x\in G$ is an element of finite order, which is necessarily
semi-simple. 
The minimal polynomial of $x$ over $\Q$ has the form
\begin{equation}
  \label{eq:gal_coh.5}
  P_{\ul n}(t)=\Phi_{n_1}(t) \Phi_{n_2}(t) \cdots \Phi_{n_r}(t), \quad
  1\le n_1<n_2<\dots <n_r
\end{equation}
for some $r$-tuple $\ul n=(n_1,\dots, n_r)\in \bbN^r$, where $\Phi_m(t)\in \Z[t]$ denotes the $m$-th
cyclotomic polynomial.   We define
\[ K_{\ul n}:=\frac{\Q[t]}{\prod_{i=1}^r \Phi_{n_i}(t)} \quad
\text{and}\quad  
   A_{\ul n}:=\frac{\Z[t]}{\prod_{i=1}^r \Phi_{n_i}(t)}. \]
The $\Q$-subalgebras of $\End^0(\ol X_0)=\Mat_d(D)$ generated by $x$ and $\pi_x=x\pi_{X_0}$
coincide and are isomorphic to $K_{\ul n}$. Moreover, the subring
$\Z[x]\subset \Mat_d(\calO)$ is canonically isomorphic to $A_{\ul
  n}$. 

We denote by $C(\ul n)\subset \Cl_0(G)$ the set of conjugacy classes of
$G$ with minimal polynomial $P_{\ul n}(t)$.  By Theorem~\ref{thm:galois-descent-isog-classes}, each conjugacy class
$[x]\in C(\ul n)$ determines a (conjugacy class of) supersingular multiple Weil $q$-number
$\pi_1^{m_1}\times \cdots \times \pi_r^{m_r}$, where $\pi_i=(-p)^{a/2}\zeta_{n_i}$, and $\ul m=(m_1, \cdots, m_r)$ is the
type of the faithful $(K_{\ul n},D)$-bimodule structure on $V=D^d$
equipped by $\pi_x\in \Mat_d(D)$. Since
$\Q(\pi_x)=\Q(x)\cong K_{\ul n}$, the $(K_{\ul n},D)$-bimodule
structure on $V$ is also equipped \emph{directly} by $x\in\Mat_d(D)$. Thus
$\ul m$ is also called the type of $[x]$, as it depends only on
the conjugacy class.  Recall that a $(K_{\ul n},D)$-bimodule $V$ is said to be type
$\ul m$ if the decomposition into $D$-subspaces $V=\oplus_{i=1}^r
V_i$ induced from the decomposition $K_{\ul n}=\prod_{i=1}^r
\Q(\zeta_{n_i})$ satisfies that $\dim_D V_i=m_ie(\pi_i)$ for all $1\leq i \leq r$, 
where $e(\pi_i)$ is defined in Lemma~\ref{lem:minimal-embedding}. 
Since $\dim E_0=1$, we have
$e(\pi_i)=d(\pi_i)$, the dimension of the Weil number
$\pi_i$. Note that $d(\pi_i)$ depends
only on the integer $n_i$ as $q=p^a$ is fixed, so we write $d(n_i)$ for it
instead.  Equation (\ref{eq:admissible-cond}) becomes 
\begin{equation}
  \label{eq:gal_coh.8}
  m_1 d(n_1)+\dots+m_r d(n_r)=d.
\end{equation}


A pair of $r$-tuples $(\ul n,\ul m)\in \bbN^r\times \bbN^r$
with $1\le n_1<\dots<n_r$ is said to be {\it $d$-admissible} if the
condition~(\ref{eq:gal_coh.8}) 
is satisfied. Let $C(\ul n, \ul m)\subset
C(\ul n)$ denote the subset of conjugacy classes of type $\ul
m$. An element $x\in G$ or its conjugacy $[x]\in \Cl(G)$ is said to be
type $(\ul n, \ul m)$ if $[x]\in 
C(\ul n, \ul m)$. 

\begin{lemma}\label{gal.emb}
Fix a faithful $(K_{\ul n},D)$-bimodule $V=D^d$ of type $\ul m$.  
There is a natural bijection between the set 
$C(\ul n, \ul m)$ and the set of
isomorphism classes of $(A_{\ul n},\calO)$-lattices in $V$.   
\end{lemma}
\begin{proof}
  Let $M_0:=\calO^d\subset V $ be the standard lattice in $V$.
  Every element $x\in G$ of type $(\ul n, \ul m)$  gives rise to an
  $(A_{\ul n}, \calO)$-bimodule structure on $M_0$. Two elements $x,
  x'$ determine isomorphism bimodule structures if and only if
  they are conjugate in $G$. Therefore, the set $C(\ul n, \ul m)$ is
  in bijection 
    with the set of isomorphism classes of $(A_{\ul n},
    \calO)$-lattices in $V$ that are $\calO$-isomorphic to
  $M_0$. Since $d\ge 
    2$, every $\calO$-lattice in $V$ is isomorphic to $M_0$. This
    follows from a theorem of Eichler \cite{Eichler1938} that the
  class number of 
    $\Mat_d(\calO)$ is $1$ for $d\ge 2$ (see also
    \cite[Theorem~2.1]{IbuKatOort1986}).
\end{proof}

\begin{thm}\label{gal_coh.3}
  Let $\Cl_0(G)$ be the set of conjugacy classes of
  $G=\GL_d(\calO)$ of finite order with $d\ge 2$. Then 
  \begin{equation}
    \label{eq:gal_coh}
    \Cl_0(G)=\coprod_{(\ul n, \ul m)} C(\ul n, \ul m),
  \end{equation}
  where $(\ul n, \ul m)$ runs through all $d$-admissible types. 
For each fixed $(\ul n, \ul m)$,  there are natural bijections between
the following sets:  
\begin{enumerate}
\item $C(\ul n,\ul m)$, the set of conjugacy classes of type $(\ul n,
  \ul m)$;
\item $\Sp(\pi)$,  where $\pi=\pi_1^{m_1}\times \cdots \times
  \pi_r^{m_r}$ and $\pi_i=(-p)^{a/2}\zeta_{n_i}$;   
\item the set of isomorphism classes of $(A_{\ul n},
  \calO)$-lattices in the $(K_{\ul n}, D)$-bimodule $V$ of 
  type $\ul m$.  
\end{enumerate}
\end{thm}
\begin{proof}
The bijection between (1) and (2) is
  established by combining  (\ref{eq:gal.3}) and
  Theorem~\ref{thm:galois-descent-isog-classes}.   The bijection between (1) and (3) follows from
  Lemma~\ref{gal.emb}. 
\end{proof}

\section{Arithmetic results}
\label{sec:arithmetic-results}
In this section, we prove the arithmetic results used in Section~\ref{sec:sp}
concerning the order $\calR_{sp}$.  In the light of
(\ref{eq:sp.5}), our goals are two fold: (1) show that  $\calR_{sp}$ is Bass for every supersingular multiple Weil
$p$-number $\pi\in MW_p^{\rm
  ss}(2)$ of dimension $2$ distinct from $\pm\sqrt{p}$; (2) classify all suporders of
$\calR_{sp}$ (i.e., orders in $K$ containing $\calR_{sp}$) and calculate their class numbers when $\pi$ is not of
the form $\pi_1\times \pi_1$ with $\pi_1\in
W_p^{\rm ss}(1)$ (The case $\pi=\pi_1\times \pi_1$ has already been
treated in Section~\ref{sec:sp.2}).  
\subsection{Orders in products of number fields}
Let $K=\prod_{i=1}^r K_i$ be a product of number fields, and
$\calS$ be an order contained in the maximal order
$O_K=\prod_{i=1}^rO_{K_i}$. We write $\eta_i: K\to K_i$ for the
projection map onto the $i$-th factor. 
By a theorem of Borevich and Faddeev \cite{MR0205980}
  (see \cite[Section 37, p.~789]{curtis-reiner:1} or
  \cite[Theorem 2.1]{MR794792}), $\calS$ is
  Bass if and only if $O_K/\calS$ is cyclic as an
$\calS$-module.  This leads to the following simple criterion when
 $r=2$. 

 \begin{lem}\label{lem:criterion-Bass-prod}
A suborder $\calS\subseteq O_{K_1}\times O_{K_2}$ that projects
   surjectively onto both factors $O_{K_1}$ and $O_{K_2}$ is
   Bass.
 \end{lem}
 \begin{proof}
 Each $O_{K_i}$ is equipped
   with an $\calS$-module structure via the projection map $\eta_i:
   \calS\to O_{K_i}$. 
Since $\eta_2(\calS)=O_{K_2}$, the
   natural inclusion $O_{K_1}\hookrightarrow O_{K_1}\times
   O_{K_2}$ defined by $x\mapsto (x,0)$ induces an isomorphism 
of $S$-modules 
\begin{equation}
  \label{eq:isom-module}
O_{K_1}/(O_{K_1}\cap \calS)\xrightarrow{\simeq} (O_{K_1}\times
O_{K_2})/\calS.   
\end{equation}
The left hand side is a cyclic $\calS$-module because
$\eta_1(\calS)=O_{K_1}$. 
 \end{proof}


We return to the general case with $r\geq 1$. 
Let $\a$ be an
$O_K$-lattice (i.e., a fractional $O_K$-ideal that contains a $\qq$-basis of
$K$) contained in $\calS$. There is a one-to-one correspondence
between the orders $B$ intermediate to $\calS \subseteq O_K$ and the
subrings of $O_K/\a$ containing $\calS/\a$. By \cite[Theorem
I.12.12]{MR1697859}, the class number $h(B)$ can be calculated by 
\begin{equation}
  \label{eq:classNo-order}
h(B)=\frac{h(O_K)[(O_K/\a)^\times:(B/\a)^\times]}{[O_K^\times:B^\times]},  
\end{equation}
where $h(O_K)=\prod_{i=1}^r h(O_{K_i})$. A priori,   \cite[Theorem
I.12.12]{MR1697859} is only stated for the number field case with $\a$
being the conductor of $B$, but the
same proof applies in the current setting as well. 

\begin{lem}\label{lem:classno-suborder=maxorder}
Let $\a\subset O_K$ be an $O_K$-lattice. If the natural map
$O_K^\times\to (O_K/\a)^\times$ is surjective, 
  then $h(B)=h(O_K)$ for every suborder $B \subseteq O_K$ containing $\a$. 
\end{lem}
\begin{proof}
  Let $\grK$ be the kernel of $O_K^\times\to (O_K/\a)^\times$. Then
  $\grK\subseteq B^\times$ and $[O_K^\times: B^\times]=[O_K^\times/\grK:
  B^\times/\grK]$. We identify $O_K^\times/\grK$ with the image of
  $O_K^\times\to (O_K/\a)^\times$, and similarly for $B^\times/\grK$.
  By \cite[Lemma~2.7]{xue-yang-yu:ECNF},  
  $B^\times=O_K^\times\cap B$. Hence \[B^\times/\grK=(O_K^\times/\grK)\cap
  (B/\a).\]
When $O_K^\times$ maps surjectively onto  $(O_K/\a)^\times$, we have
$B^\times/\grK=(O_K/\a)^\times\cap 
  (B/\a)=(B/\a)^\times$. Therefore, $h(B)=h(O_K)$ by (\ref{eq:classNo-order}). 
\end{proof}
 \begin{lem}\label{lem:ideal-in-Rsp}
   Let $\calS$ be a suborder of $O_K=\prod_{i=1}^r O_{K_i}$, and $\c_1$ be a
   nonzero ideal of $O_{K_1}$ contained in $\eta_1(\calS)$.  If $x_1\in
   O_{K_1}$ is an element such that $(x_1,0, \cdots, 0)\in \calS$,
   then $(x_1\c_1, 
   0, \cdots, 0)$ is an ideal of  $O_K$ contained in $\calS$.  Similar
   results hold for all $1\leq i \leq r$. 
 \end{lem}
 \begin{proof}
 Clearly $(x_1\c_1,
   0, \cdots, 0)$ is an ideal of $O_K$.   For any element $y_1\in
   \c_1$, we may find $\mathbf{y}\in \calS$ such 
   that $\eta_1(\mathbf{y})=y_1$. Then 
$(x_1y_1,0,\cdots,0)=(x_1,0,\cdots,0)\cdot \mathbf{y}\in \calS$. 
 \end{proof}

\subsection{The order $\calR_{sp}$ is Bass when
  $d(\pi)=2$.}\label{subsec:order-Rsp-bass} 
We recall the definition of $\calR_{sp}$. Suppose that
$\pi=\pi_1^{m_1}\times \cdots \times \pi_r^{m_r}$ is a supersingular
multiple Weil $p$-number with $m_i\in \bbN$ and $\pi_i\not \sim \pi_j$.
Let $K=\prod_i K_i$ with $K_i=\qq(\pi_i)$, and
$\pi_0=(\pi_1, \ldots, \pi_r)\in K$. Then $\calR_{sp}$ is defined to
be the order $ \zz[\pi_0, \pi_0^2/p]\subseteq O_K$. Assume that $\pi$
has dimension 2 and none of $\pi_i$ is conjugate to $\sqrt{p}$. The
case $\pi=\pi_1^2$ with $\pi_1\in W^{\rm ss}_p(1)$ has already been
studied in Section~\ref{sec:sp.2}.  It remains to treat the following two
cases:
\begin{enumerate}
\item[(1)]  $\pi=\pi_1\times \pi_2$ with both $\pi_1, \pi_2\in
W_p^{\rm ss}(1)$ and $\pi_1\not\sim \pi_2$ (the nonisotypic product case); 
\item[(2)] $\pi=\pi_1\in W_p^{\rm
  ss}(2)$ and $\pi_1\not\sim \sqrt{p}$ (the nonreal simple case). 
\end{enumerate}
The first case occurs only when 
\begin{equation}
  \label{eq:1}
p=2, 3,\quad\text{ and } \quad \pi=\sqrt{p}\zeta_4
\times (\pm \sqrt{p}\zeta_{4p}),\quad \text{or}\quad \sqrt{p}\zeta_{4p}\times
(-\sqrt{p}\zeta_{4p}).  
\end{equation}
In the second case, the supersingular Weil $p$-numbers of
 dimension 2 distinct from $\pm \sqrt{p}$ are 
 \begin{equation}
   \label{eq:2}
   \sqrt{p} \zeta_3,\ \pm \sqrt{p} \zeta_5\ (p=5),
\  \sqrt{p}\zeta_8\ (p\neq 2),\  \sqrt{p}\zeta_{12} \ (p\neq3),
\ \pm\sqrt{p}\zeta_{24}\ (p=2). 
 \end{equation}

 \begin{lem}\label{lem:Rsp-product-case}
Assume $p=2$ or $3$. 
If $\pi=\sqrt{p}\zeta_4
\times (\pm \sqrt{p}\zeta_{4p})$, then 
\[\calR_{sp}=\zz[(\sqrt{-p}, 0),
(0, 1+\zeta_{2p})]\subset \qq(\sqrt{-p})\times\qq(\zeta_{2p})=K.\]
If $\pi=\sqrt{p}\zeta_{4p}\times (-\sqrt{p}\zeta_{4p})$, then 
\[\calR_{sp}=\zz[(2(1+\zeta_{2p}, 0),
(\zeta_{2p}, \zeta_{2p})]\subset \qq(\zeta_{2p})\times\qq(\zeta_{2p})=K.\]
 \end{lem}
 \begin{proof}
Note that $\sqrt{p}\zeta_{4p}=1+\zeta_{2p}$ when $p=2$ or $3$. If $\pi=
\sqrt{p}\zeta_4
\times (\pm \sqrt{p}\zeta_{4p})$, then
\[
\begin{split}
  \calR_{sp}&=\zz[(\sqrt{-p},\pm \sqrt{p}\zeta_{4p}), (-1,
  \zeta_{2p})]=\zz[(\sqrt{-p},\pm \sqrt{p}\zeta_{4p}), (0,
  1+\zeta_{2p})]\\
 &=\zz[(\sqrt{-p}, 0),
(0, 1+\zeta_{2p})]. 
\end{split}
\]
If $\pi=\sqrt{p}\zeta_{4p}\times (-\sqrt{p}\zeta_{4p})$, then 
\[
\begin{split}
  \calR_{sp}&=\zz[(\sqrt{p}\zeta_{4p}, -\sqrt{p}\zeta_{4p}),
  (\zeta_{2p}, \zeta_{2p})]=\zz[(1+\zeta_{2p}, -(1+\zeta_{2p})),
  (\zeta_{2p}, \zeta_{2p})]\\
 &=\zz[(2(1+\zeta_{2p}), 0), (\zeta_{2p}, \zeta_{2p})]. \qedhere
\end{split}
\]
 \end{proof}


 \begin{prop}
The order $\calR_{sp}$ is a Bass order for every supersingular multiple Weil
$p$-number $\pi \in MW_p^{\rm ss}(2)$ distinct from $\pm\sqrt{p}$. 
 \end{prop}
 \begin{proof} We only need to consider the cases where $\pi$ is not
   of the form $\pi_1^2$ with $\pi_1\in W_p^{\rm ss}(1)$. 
    Suppose that
   $\pi=\pm \sqrt{p}\zeta_n\in W_p^{\rm ss}(2)$ is one of the Weil
   $p$-numbers listed in (\ref{eq:2}), and $m$ is defined as in
   (\ref{eq:def-of-m}).
   If $n$ is critical at $p$, then $\calR_{sp}$ equals to the maximal order $\Z[\zeta_m]$
   in $K=\Q(\zeta_m)$ by Remark~\ref{rem:critical-max-order}.
    Otherwise, $[K:\qq(\zeta_m)]=2$ and $\calR_{sp}$ is a
   quadratic $\zz[\zeta_m]$-order, and such type of orders are Bass 
   \cite[Example 2.3]{MR794792}.   

If $p=2,3$ and $\pi=\sqrt{p}\zeta_{4p}\times (-\sqrt{p}\zeta_{4p})$,
or $p=2$ and $\pi=\sqrt{2}\zeta_4\times (\pm \sqrt{2}\zeta_8)$, then
$\calR_{sp}$
projects surjectively onto both $O_{K_1}$ and $O_{K_2}$, and hence $\calR_{sp}$ is Bass by Lemma~\ref{lem:criterion-Bass-prod}. 

Lastly, suppose that $p=3$ and
$\pi=\sqrt{3}\zeta_4\times (\pm \sqrt{3}\zeta_{12})$. Then
$\eta_1(\calR_{sp})=\zz[\sqrt{-3}]$, a suborder of index 2 in
$O_{K_1}=\zz[\zeta_6]$, while $\eta_2(\calR_{sp})=\zz[\zeta_6]=O_{K_2}$. So by
(\ref{eq:isom-module}), to show that $\calR_{sp}$ is Bass, it
is enough to prove that $O_{K_1}/(O_{K_1}\cap \calR_{sp})$ is a cyclic
$\calR_{sp}$-module. Note that $O_{K_1}\subset O_{K_1}\times O_{K_2}$ is generated
by $(1,0)$ and $(\zeta_6,0)$ over $\zz$, and 
\[\calR_{sp}(\zeta_6, 0)\ni (-1+\sqrt{-3},
-1)\cdot(\zeta_6,0)=(1,0)+(\sqrt{-3},0)^2\equiv (1,0)
\pmod{O_{K_1}\cap\calR_{sp}}. \] 
Hence  $O_{K_1}/(O_{K_1}\cap \calR_{sp})$ is a cyclic
$\calR_{sp}$-module generated by $(\zeta_6, 0)$. 
 \end{proof}

\subsection{Suporders of $\calR_{sp}$ and class numbers: the
  nonisotypic product case}\label{subsec:supord-Rsp-prod-case} Assume that
$p=2$ or $3$ and $\pi=\pi_1\times \pi_2$ is a supersingular multiple
Weil $p$-number of 
dimension 2 listed in (\ref{eq:1}).  Using
Lemma~\ref{lem:ideal-in-Rsp} and Lemma~\ref{lem:Rsp-product-case}, one
may easily find an $O_K$-lattice $\a$ contained in $\calR_{sp}$ and
compute the quotient rings $O_K/\a$ and $\calR_{sp}/\a$. We obtain the
following 
table (For simplicity, we set $i=\zeta_4=\sqrt{-1}$). 

\smallskip 
\renewcommand{\arraystretch}{1.1}
\begin{tabular}{|c|c|c|c|}
\hline
$\pi=\pi_1\times \pi_2$ &   $\a\subset \calR_{sp}$ & $O_K/\a$ &
$\calR_{sp}/\a$ \\ \hline 
$ \sqrt{2}\zeta_4\times \pm \sqrt{2}\zeta_8$  
& $\sqrt{-2}O_{K_1}\times (1+i)O_{K_2}$ & $(\F_2)^2$  &  $\scrD_2$ \\  \hline
$\sqrt{2}\zeta_8\times - \sqrt{2}\zeta_8$ & 
                                       $(2(1+i)O_{K_1})^2$  &
$(\Z[i]/(1+i)^3)^2$ & $\scrD_8$ \\ \hline
$ \sqrt{3}\zeta_4\times \pm \sqrt{3}\zeta_{12} $  & 
 $(2\sqrt{-3})O_{K_1}\times \sqrt{-3}O_{K_2}$ &
$\F_4\times (\F_3)^2$ &
$\F_2\times \scrD_3$\\ \hline
$ \sqrt{3}\zeta_{12}\times -\sqrt{3}\zeta_{12}$ &
 $(2\sqrt{-3})O_{K_1}\times (2\sqrt{-3})O_{K_2}$ &
 $(\F_4\times \F_3)^2$ &
 $\scrD_{12}$ \\ \hline
\end{tabular} \\[3pt]
\renewcommand{\arraystretch}{1}
Here $\scrD_2$, $\scrD_8$, $\scrD_3$, and $\scrD_{12}$ denote
the
diagonal in $(\F_2)^2$,  $(\Z[i]/(1+i)^3)^2$, $(\F_3)^2$, and
$(\F_4\times\F_3)^2$ respectively.  


It is an exercise to show that
$O_K^\times$ maps surjectively onto $(O_K/\a)^\times$ in all the above
cases. By Lemma~\ref{lem:classno-suborder=maxorder},  $h(B)=h(O_K)$
for every order $B$ with $\calR_{sp}\subseteq B \subseteq O_K$. Note
that $h(O_K)=h(O_{K_1})h(O_{K_2})=1$ since both $\Z[i]$ and
$\Z[\zeta_6]$ have class number 1.  We obtain the following
proposition. 
\begin{prop}
  Assume that
$p=2$ or $3$ and $\pi=\pi_1\times \pi_2$ is given in
(\ref{eq:1}). Then any suporder $B$ of $\calR_{sp}$ has class number $1$. 
\end{prop}

It remains to list all suporders $B$ of $\calR_{sp}$ for each
$\pi$. We recall the convention in Section~\ref{sec:sp.2} that a suporder of
$\calR_{sp}$ with index $j>1$ in $O_K$ is denoted by $B_j$. Our
calculation will show that for those $\pi$ considered in this
subsection, if such an order exists, then it is
unique.  So there is no ambiguity in this notation if $\pi$ is clear
from the context.  We separate into cases. 

\smallskip

\noindent\textbf{Case} $\pi=\sqrt{2}\zeta_4\times \pm
\sqrt{2}\zeta_8$. Since $[O_K:\calR_{sp}]=[O_K/\a:\calR_{sp}/\a]=2$,
there are no other suporders of $\calR_{sp}$ besides $\calR_{sp}$
and $O_K$.  

\noindent\textbf{Case} $\pi=\sqrt{3}\zeta_4\times \pm
\sqrt{3}\zeta_{12}$. We have $[O_K:\calR_{sp}]=[\F_4\times (\F_3)^2:
\F_2\times \scrD_3]=6$. There are two rings properly intermediate to the
inclusion $\F_2\times \scrD_3 \subset \F_4\times (\F_3)^2$, namely
$\F_4\times \scrD_3$ and $\F_2\times (\F_3)^2$.  Under the inclusion-preserving
correspondence between suborders of $O_K$ containing $\a$ and subrings
of $O_K/\a$, we have
\begin{align*}
  B_3:=\Z[(\sqrt{-3},0), (\zeta_6, \zeta_6)]=\Z[(1+\zeta_6, 0), (0, 1+\zeta_6)]&\longleftrightarrow \F_4\times
                                          \scrD_3,\\
  B_2:=\Z[\sqrt{-3}]\times \Z[\zeta_6]&\longleftrightarrow \F_2\times (\F_3)^2.
\end{align*}

The remaining two cases are best seen in the light of the following
lemma. 
\begin{lem}\label{lem:subrings-containing-diagonal}
  Let $R$ be a commutative ring, and $\scrD$ be the diagonal of $R^2$.
  Every subring $S$ of $R^2$  containing $\scrD$ decomposes uniquely as
  $\scrD\oplus (I_S, 0)$, where $I_S$ is an ideal of $R$.  In
  particular, there is an inclusion-preserving bijective
  correspondence between subrings of $R^2$ containing $\scrD$ and
  ideals of $R$. 
\end{lem}
\begin{proof}
  Every subring of $R^2$ containing $\scrD$ is naturally an
  $R$-submodule of $R^2$.  So the intersection $(I_S,0):=S\cap (R, 0)$
  is again an $R$-submodule of $R^2$. Equivalently, $I_S$ is an ideal
  of $R$.  Clearly, we have $S=\scrD\oplus (I_S,0)$. Conversely, for any ideal $I\subseteq R$, the direct sum
  $\scrD\oplus (I,0)$ is a subring of $R^2$.  The correspondence is
  established. 
\end{proof}

By Lemma~\ref{lem:Rsp-product-case}, if $\pi=\sqrt{p}\zeta_{4p}\times
-\sqrt{p}\zeta_{4p}$ with $p=2$ or $3$, then
$O_K=\Z[\zeta_{2p}]^2$, and 
\[\calR_{sp}=\Z[(2(1+\zeta_{2p}), 0),
(\zeta_{2p}, \zeta_{2p})]=\scrD\oplus
(2(1+\zeta_{2p})\Z[\zeta_{2p}],0).\]

\noindent\textbf{Case} $\pi=\sqrt{2}\zeta_8\times -
\sqrt{2}\zeta_8$.  We have $2(1+i)\Z[i]=(1+i)^3\Z[i]$. So by
Lemma~\ref{lem:subrings-containing-diagonal}, the suborders of $O_K$
properly containing $\calR_{sp}$ and distinct from $O_K$ are 
\begin{align*}
  B_4:=\Z[(i,i), (2,0)] &\longleftrightarrow  (1+i)^2\Z[i]=2\Z[i],\\
  B_2:=\Z[(i,i), (1+i,0)] &\longleftrightarrow  (1+i)\Z[i].
\end{align*}
\noindent\textbf{Case} $\pi=\sqrt{3}\zeta_{12}\times -
\sqrt{3}\zeta_{12}$.  In this case, the ideal 
$2(1+\zeta_6)\Z[\zeta_6]$ factors as the product of the prime ideals
$2\Z[\zeta_6]$ and $\sqrt{-3}\Z[\zeta_6]$. The suborders of $O_K$
properly containing $\calR_{sp}$ and distinct from $O_K$ are 
\begin{align*}
  B_4:=\Z[(\zeta_6,\zeta_6), (2,0)] &\longleftrightarrow  2\Z[\zeta_6],\\
  B_3:=\Z[(\zeta_6,\zeta_6), (\sqrt{-3},0)] &\longleftrightarrow  \sqrt{-3}\Z[\zeta_6].
\end{align*}

\subsection{Suporders of $\calR_{sp}$ and class numbers: the nonreal
  simple case}\label{subsec:supord-Rsp-simple-case}
 Assume that $\pi$ is a supersingular Weil $p$-number of
dimension $2$ listed in (\ref{eq:2}). Only the case
$\pi=\sqrt{p}\zeta_{12}$ with $p\neq 3$ needs to be studied, as the
 rest have already been covered in
 Section~\ref{sec:sp.2}. 

 If $\pi=\sqrt{p}\zeta_{12}$, we have $K=\Q(\sqrt{-p}, \sqrt{-3})$,
 and $\calR_{sp}=\Z[\sqrt{-p}, \zeta_6]$. Since the discriminants of
 $\Q(\sqrt{-3})$ and $\Q(\sqrt{-p})$ are coprime, $O_K$ is the
 compositum of $\Z[\zeta_6]$ and $O_{\Q(\sqrt{-p})}$.  If $p=2$ or
 $p\equiv 1\pmod{4}$, then $O_{\Q(\sqrt{-p})}=\Z[\sqrt{-p}]$, and
 $\calR_{sp}$ is the maximal order in $K$. We assume that
 $p\equiv 3\pmod{4}$ and $p\neq 3$ for the rest of this
 subsection. Note that $2O_K\subseteq \calR_{sp}$, and 
$\calR_{sp}/2O_K=\Z[\zeta_6]/(2)\simeq \F_4$, which embeds into
$O_K/2O_K\simeq \F_4\oplus \F_4$ diagonally. It follows that
$\calR_{sp}$ and $O_K$ are the only orders in $O_K$ containing 
$\calR_{sp}$. By (\ref{eq:classNo-order}),
$h(\calR_{sp})=3h(O_K)/[O_K^\times : \calR_{sp}^\times]$. It remains
to calculate the index $[O_K^\times : \calR_{sp}^\times]$. 







 \begin{lem}
   Let $p_1$ and $p_2$ be distinct primes with $p_1\equiv p_2 \equiv 3
   \pmod{4}$, and $\epsilon$ be the fundamental unit of
   $F=\qq(\sqrt{p_1p_2})$.  Then $\sqrt{-\epsilon}\in
   K=\qq(\sqrt{-p_1}, \sqrt{-p_2})$, and
   $O_K^\times=\dangle{\sqrt{-\epsilon}}\times \bmu_K$,  the direct
   product of the free abelian group generated by $\sqrt{-\epsilon}$
   and the group $\bmu_K$ of roots of unity in $K$.   Moreover, if $\epsilon\in
   \zz[\sqrt{p_1p_2}]$, then $\sqrt{-\epsilon}$ lies in the
   $\zz$-module $\zz\sqrt{-p_1}+\zz\sqrt{-p_2} \subset O_K$; otherwise
   $\sqrt{-\epsilon}\equiv (\sqrt{-p_1}+\sqrt{-p_2})/2
   \pmod{\zz\sqrt{-p_1}+\zz\sqrt{-p_2}}$. 
 \end{lem}

 \begin{proof}
By Dirichlet's Unit Theorem, the quotient group $O_K^\times/\bmu_K$ is a free abelian group
of rank 1 containing $O_F^\times/\{\pm 1\}\cong \dangle{-\epsilon}$ as a subgroup of finite
index. In fact, we have $[O_K^\times/\bmu_K: O_F^\times/\{\pm 1\}]\leq
2$ by \cite[Theorem 4.12]{Washington-cyclotomic} as $K$ is a
CM-field with maximal totally real subfield $F$. 

Since both $p_i\equiv 3 \pmod{4}$, it follows from
\cite[(V.1.7)]{ANT-Frohlich-Taylor} that the norm $\Nm_{F/\qq}(\epsilon)=+1$.
   By \cite[Lemma 3]{MR0441914},  $p_1\epsilon$ is a perfect square in $F^\times$. 
   Write $p_1\epsilon=(x+y\sqrt{p_1p_2})^2$ with $x,y\in \qq$. Then 
   \[\sqrt{-\epsilon}=\sqrt{p_1\epsilon}\cdot
   \frac{-1}{\sqrt{-p_1}}=(x+y\sqrt{p_1p_2})\cdot
   \frac{-1}{\sqrt{-p_1}}\in \qq\sqrt{-p_1}+
   \qq\sqrt{-p_2}\subset K.\]
In particular, $[O_K^\times/\bmu_K: O_F^\times/\{\pm 1\}]\geq 2$. It follows that $[O_K^\times/\bmu_K: O_F^\times/\{\pm
1\}]=2$, and $O_K^\times/\bmu_K\cong \dangle{\sqrt{-\epsilon}}$. Hence
$O_K^\times =\dangle{\sqrt{-\epsilon}}\times \bmu_K$. 

By our assumption on $p_i$,  the prime $2$ is unramified
in $O_K$. One easily checks that the following statements are equivalent:
   \begin{enumerate}
   \item $\epsilon\in \zz[\sqrt{p_1p_2}]=\zz+2O_F$; 
   \item $\epsilon\equiv 1 \pmod{2O_F}$;
   \item $\sqrt{-\epsilon}\equiv 1
   \pmod{2O_K}$;
 \item $\sqrt{-\epsilon}\in \zz+2O_K$. 
   \end{enumerate}

By Exercise~42(d) of \cite[Chapter~2]{MR0457396}, a $\zz$-basis of
$O_K$ is given by  
\[\left\{1, \quad\frac{1+\sqrt{-p_1}}{2}, \quad
\frac{1+\sqrt{-p_2}}{2},\quad \frac{(1+\sqrt{-p_1})(1+\sqrt{-p_2})}{4}
\right\}.\]
It follows that 
\begin{gather*}
  O_K\cap (\qq\sqrt{-p_1}+
   \qq\sqrt{-p_2})=\zz\sqrt{-p_1}+
   \zz(\sqrt{-p_1}+\sqrt{-p_2})/2;\\
(\zz+2O_K)\cap (\qq\sqrt{-p_1}+
   \qq\sqrt{-p_2})=\zz\sqrt{-p_1}+
   \zz\sqrt{-p_2}.
\end{gather*}
Therefore, if $\epsilon\in \zz[\sqrt{p_1p_2}]$, then
$\sqrt{-\epsilon}\in \zz\sqrt{-p_1}+
   \zz\sqrt{-p_2}$. Otherwise 
$\sqrt{-\epsilon}$ lies in $\zz\sqrt{-p_1}+
   \zz(\sqrt{-p_1}+\sqrt{-p_2})/2$ but not in $\zz\sqrt{-p_1}+
   \zz\sqrt{-p_2}$. Hence $\sqrt{-\epsilon}\equiv (\sqrt{-p_1}+\sqrt{-p_2})/2
   \pmod{\zz\sqrt{-p_1}+\zz\sqrt{-p_2}}$ in this case. 
 \end{proof}

We return to the assumption that $K=\Q(\sqrt{-p}, \sqrt{-3})$ with $p\equiv
3\pmod{4}$ and $p\neq 3$. Note that $\bmu_K=\dangle{\zeta_6}\subset
\calR_{sp}^\times$, and $\calR_{sp}\cap
(\Q\sqrt{-p}+\Q\sqrt{-3})=(\Z\sqrt{-p}+\Z\sqrt{-3})$.  Let $\epsilon$ be the fundamental unit of
$F=\Q(\sqrt{3p})$. If $\epsilon\in \Z[\sqrt{3p}]$, then
$\sqrt{-\epsilon}\in \calR_{sp}$, and hence
$\calR_{sp}^\times=O_K^\times$.  This holds in particular when
$p\equiv 3 \pmod{8}$ and $p\neq 3$ as remarked after (\ref{eq:varpi_d}). 
 Assume that $p\equiv 7\pmod{8}$ and
$\epsilon\not\in\Z[\sqrt{3p}]$. Then  $(O_F/2O_F)^\times \simeq
\F_4^\times$ and $\epsilon^3\in \Z+2O_F=\Z[\sqrt{3p}]$. On the
other hand, $[(O_K/2O_K)^\times:
(\calR_{sp}/2O_K)^\times]=[(\F_4^\times)^2:\F_4^\times]=3$, so we have $\sqrt{-\epsilon}\not\in \calR_{sp}$ but
$(\sqrt{-\epsilon})^3\in \calR_{sp}$. 

In summary, we find that 
\[[O_K^\times: \calR_{sp}^\times]=[O_F^\times :\Z[\sqrt{3p}]^\times]=
\begin{cases}
  1 &\quad \text{if } \epsilon \in \Z[\sqrt{3p}];\\
  3 & \quad \text{otherwise.} 
\end{cases}\]
Therefore, we have $h(\calR_{sp})=\varpi_{3p}h(O_K)$, where $\varpi_{3p}=3/[O_F^\times
:\Z[\sqrt{3p}]^\times]$ as defined in (\ref{eq:varpi_d}).




\section*{Acknowledgements}
J.~Xue is partially supported
by the 1000-plan program for young scholars of PRC. He thanks Academia
Sinica and the NCTS for their warm hospitality and great working
conditions. TC Yang and CF Yu 
are partially supported by the grants MoST 100-2628-M-001-006-MY4,
103-2811-M-001-142, 104-2115-M-001-001MY3 and 104-2811-M-001-066.

\bibliographystyle{plain}
\bibliography{TeXBiB}
\end{document}